\documentclass[10pt]{article}
\usepackage{amssymb,amsmath,amsfonts,amsthm}
\usepackage{graphicx}
\usepackage[a4paper,includeheadfoot,left=3.95cm,right=3.95cm,top=2.8cm,bottom=3.0cm]{geometry}

\usepackage[font={scriptsize}]{caption}

\usepackage[numbers]{natbib}
\makeatletter\let\@fnsymbol\@arabic\makeatother

\usepackage{color}
\definecolor{marin}{rgb}{0.,0.3,0.7}
\usepackage[colorlinks,citecolor=marin,linkcolor=marin,urlcolor=marin,bookmarksopen,bookmarksnumbered]{hyperref}

\newcommand{\al}{\alpha}
\newcommand{\be}{\beta}

\newcommand{\si}{\sigma}
\newcommand{\ta}{\tau}
\newcommand{\om}{\omega}
\newcommand{\Om}{\Omega}
\newcommand{\Z}{\mathbb{Z}}
\newcommand{\R}{\mathbb{R}}
\newcommand{\C}{\mathbb{C}}
\newcommand{\TT}{\mathbb{T}}
\providecommand{\abs}[1]{\lvert#1\rvert}
\providecommand{\absbig}[1]{\bigl\lvert#1\bigr\rvert}
\providecommand{\absBig}[1]{\Bigl\lvert#1\Bigr\rvert}
\providecommand{\absbigg}[1]{\biggl\lvert#1\biggr\rvert}
\providecommand{\norm}[1]{\lVert#1\rVert}
\providecommand{\normbig}[1]{\bigl\lVert#1\bigr\rVert}

\providecommand{\normbigg}[1]{\biggl\lVert#1\biggr\rVert}
\providecommand{\normv}[1]{\ensuremath{{\lVert\hskip-1pt\lvert}#1{\rvert\hskip-1pt\rVert}}}
\providecommand{\normvbig}[1]{\ensuremath{{\bigl\lVert\hskip-1pt\bigl\lvert}#1{\bigr\rvert\hskip-1pt\bigr\rVert}}}

\providecommand{\skla}[1]{\langle#1\rangle}

\providecommand{\klabig}[1]{\bigl(#1\bigr)}
\providecommand{\klaBig}[1]{\Bigl(#1\Bigr)}
\providecommand{\klabigg}[1]{\biggl(#1\biggr)}
\newcommand{\formulatext}[1]{\qquad\text{#1}\qquad}
\newcommand\myfor{\formulatext{for}}
\newcommand\myand{\formulatext{and}}
\newcommand\myif{\formulatext{if}}
\newcommand\with{\formulatext{with}}
\newcommand{\sfrac}[2]{\mbox{$\textstyle\frac{#1}{#2}$}}
\newcommand{\iu}{\mathrm{i}}
\newcommand{\e}{\mathrm{e}}
\newcommand{\drm}{\mathrm{d}}
\DeclareMathOperator{\diag}{diag}
\DeclareMathOperator{\sinc}{sinc}
\DeclareMathOperator{\Id}{Id}
\newcommand{\conv}{\mathbin{\ast}}

\newtheorem{theorem}{Theorem}[section]
\newtheorem{lemma}[theorem]{Lemma}
\newtheorem{proposition}[theorem]{Proposition}
\theoremstyle{definition}
\newtheorem{assum}{Assumption}

\newcommand{\disc}{\mathcal{K}}
\newcommand{\meth}{\mathcal{S}}

\newcommand{\deltax}{{\Delta x}}

\title{Error analysis of trigonometric integrators\\ for semilinear wave equations}

\author{Ludwig Gauckler\,\thanks{Institut f\"ur Mathematik,
          Technische Universit\"at Berlin,
          Stra{\ss}e des 17.\ Juni 136,
          D-10623 Berlin, Germany
          ({\tt gauckler@math.tu-berlin.de}).}
}

\date{Version of 27 January 2015}

\begin{document}

\maketitle

\begin{abstract}
An error analysis of trigonometric integrators (or exponential integrators) applied to spatial semi-discretizations of semilinear wave equations with periodic boundary conditions in one space dimension is given. In particular, optimal second-order convergence is shown requiring only that the exact solution is of finite energy. The analysis is uniform in the spatial discretization parameter. It covers the impulse method which coincides with the method of Deuflhard and the mollified impulse method of Garc\'{i}a-Archilla, Sanz-Serna \& Skeel as well as the trigonometric methods proposed by Hairer \& Lubich and by Grimm \& Hochbruck. 
The analysis can also be used to explain the convergence behaviour of the St\"ormer--Verlet/leapfrog discretization in time.\\[1.5ex]
% Numerical Analysis ->  Partial differential equations, initial value and time-dependent initial-boundary value problems -> Method of lines
\end{abstract}

\section{Introduction}

We consider, for some integer $p\ge 2$, the semilinear wave equation 
\begin{equation}\label{eq-nlw}
u_{tt} = u_{xx} + u^p, \qquad u = u(x,t)
\end{equation}
with $2\pi$-periodic boundary conditions in one space dimension ($x\in\TT=\R/(2\pi\Z)$). Denoting by $H^s$ the Sobolev space $H^s(\TT)$, we equip this equation with initial values
\begin{equation}\label{eq-nlw-init}
u(\cdot,t_0) \in H^{s+1} \myand u_t(\cdot,t_0) \in H^s \myfor s\ge 0.
\end{equation}
We are in particular interested in the case $s=0$, where the energy is finite.

After a semi-discretization in space, this nonlinear wave equation becomes a huge system of ordinary differential equations of the form
\begin{equation}\label{eq-nlw-semi-temp}
\ddot{y} = - \Om^2 y + f(y), \qquad y = y(t)
\end{equation}
with a matrix $-\Om^2$ describing the discretized second spatial derivative in~\eqref{eq-nlw} and a nonlinearity $f(y)$ describing the polynomial nonlinearity in~\eqref{eq-nlw}. The eigenvalues of the matrix $-\Om^2$, i.e., the eigenvalues of the discretized Laplace operator, range from order one to the order of the spatial discretization parameter. 
The spatial discretization parameter is typically large, in particular for initial values of low regularity, such as \eqref{eq-nlw-init} with $s=0$, for which a very large spatial discretization parameter compensates for the slow convergence of the semi-discretization in space, see \cite{Gauckler} for the case of a spectral semi-discretization in space. 
The spatial semi-discretization thus exhibits a variety of oscillations, ranging from low to high oscillations. 

For the discretization in time of oscillatory systems of the form~\eqref{eq-nlw-semi-temp}, the use of a \emph{trigonometric integrator} (or \emph{exponential integrator}) is increasingly popular. See, for instance, \cite[Chapter~XIII]{Hairer2006} and the recent review \cite{Hochbruck2010}. These integrators are especially designed to deal with the matrix $\Om$ and the induced high oscillations. There are several papers that consider trigonometric integrators when applied to wave equations. In \cite{Cano2013,Cohen2008a}, the long-time behaviour of these methods with respect to conserved or almost conserved quantities is studied. Moreover, the methods are extended to higher order in \cite{Cano2010a,Cano2013a} and to the linear stochastic wave equation in \cite{Cohen2013}. 

To our knowledge, however, there is no rigorous error analysis of trigonometric integrators applied to spatial semi-discretizations of nonlinear wave equations such as~\eqref{eq-nlw} yet, for example for initial values of finite energy, that is~\eqref{eq-nlw-init} with $s=0$. The main challenge are error bounds that are uniform in the large frequencies and the size of the system, and hence in the spatial discretization parameter, and that allow for initial values of low regularity, such as~\eqref{eq-nlw-init} with $s=0$. 

In the present paper we prove such error bounds of trigonometric integrators applied to a spectral semi-discretization in space. We consider in particular initial values of finite energy and exact solutions $(y,\dot{y})$ of the spatial semi-discretization~\eqref{eq-nlw-semi-temp} in a discrete counterpart of $H^1\times H^0=H^1\times L_2$. Under such low regularity assumptions, we show, amongst others, second-order convergence of $(y,\dot{y})$ in $H^0\times H^{-1}$ and first-order convergence in $H^1\times H^0$. The analysis covers the impulse method \cite{Grubmueller1991,Tuckerman1992} which coincides in our situation with the method of Deuflhard \cite{Deuflhard1979} and the mollified impulse method of Garc\'{i}a-Archilla, Sanz-Serna \& Skeel \cite{GarciaArchilla1999} as well as the trigonometric methods proposed by Hairer \& Lubich \cite{Hairer2000} and by Grimm \& Hochbruck \cite{Grimm2006}. 

We mention that there are many papers that study the error of various instances of trigonometric integrators when applied to systems of the form~\eqref{eq-nlw-semi-temp}, see \cite{Cano2010a,Cano2013a,Dong2014,GarciaArchilla1999,Grimm2005,Grimm2006a,Grimm2006,Hairer2006,Hochbruck1999}. In all these works, it is assumed that the nonlinearity is in particular Lipschitz continuous. This, however, is \emph{not} the case for the nonlinear wave equation~\eqref{eq-nlw} or typical spatial semi-discretizations thereof, for example when considered in the space $H^0=L_2$ which is the natural space to prove error estimates for initial values of finite energy. An exception is the Sine--Gordon equation $u_{tt}=u_{xx}-\sin(u)$, whose nonlinearity is indeed Lipschitz continuous in $H^0=L_2$, and for which Gautschi-type trigonometric integrators have been analysed in \cite{Grimm2006a}. Another way to avoid the non-Lipschitz continuous nonlinearity is to consider equations such as~\eqref{eq-nlw} in higher order Sobolev spaces, where the nonlinearity is locally Lipschitz continuous, and where second-order error bounds can then be shown under correspondingly higher regularity assumptions on the exact solution, see \cite{Dong2014} and also \cite[Chapter IV]{Faou2012}. 

Yet another way to deal with the non-Lipschitz nonlinearity is to impose additional assumptions on the numerical solution, such as bounds in $L_\infty$ that are uniform in the time step-size. The validity of such an assumption on the numerical solution is at first not clear, however. Under such an unclear assumption, the aforementioned previous results and their proofs would also hold for non-Lipschitz nonlinearities such as $u^p$. In the present paper, we use an analysis that proves such properties of the numerical solution, notably without requiring higher regularity of the exact solution. 
This is done by exploiting the full scale of Sobolev spaces, including Sobolev spaces of negative order. More precisely, the error analysis is performed in two stages. First, a low order error bound is shown in a higher order Sobolev space (or its discrete counterpart), where the nonlinearity is, at least locally, Lipschitz continuous. From this low order error bound, a suitable regularity of the numerical solution is deduced. This regularity is then used in the second stage to overcome the lack of Lipschitz continuity in lower order Sobolev spaces and allows us to show higher order error bounds in these spaces. Such kinds of two-stage arguments have been used previously, for example in \cite{Lubich2008a,Koch2011,Gauckler2011,Thalhammer2012a} for discretizations of nonlinear Schr\"odinger equations and in \cite{Gottlieb2012,Holden2013} for discretizations of equations with Burgers nonlinearity. 

Surprisingly, it is possible to do the error analyses in both stages following the traditional argument of error accumulation in Lady Windermere's fan. This is in striking contrast to previous error analyses of trigonometric integrators given for different situations in \cite{GarciaArchilla1999,Grimm2005,Grimm2006a,Grimm2006,Hochbruck1999}, where cancellation effects in the accumulation of errors are of vital importance. 
In the case of the nonlinear wave equation, not only a conceptually different proof is possible, but also less restrictive assumptions on the filter functions that characterize the trigonometric integrator in a one-step formulation are needed. Therefore, a considerably larger class of trigonometric integrators in one-step formulation is covered by the presented analysis, in particular methods that do not use a filter inside the nonlinearity.

The paper is organised as follows. In Section~\ref{sec-method}, the considered discretization is introduced, the error bounds are stated and numerical experiments are presented. The proof of the error bounds is given in Section~\ref{sec-proof}. The presented error analysis of trigonometric integrators is not restricted to the spectral semi-discretization in space of the nonlinear wave equation~\eqref{eq-nlw} with pure power nonlinearity. It applies equally to the spectral semi-discretization of nonlinear wave equations with general polynomial or analytic nonlinearities and to the spatial semi-discretization by finite differences, as is described in Section~\ref{sec-ext}. 
Moreover, the analysis can be extended to the widely used St\"ormer--Verlet/leapfrog discretization in time by interpreting this method as a trigonometric integrator with modified frequencies, which is also described in Section~\ref{sec-ext}.

\section{Numerical method and statement of the main result}\label{sec-method}

\subsection{Spectral semi-discretization in space}\label{subsec-collocation}

For the semi-discretization in space of the nonlinear wave equation~\eqref{eq-nlw}, we consider spectral collocation. The trigonometric polynomial
\begin{equation}\label{eq-ansatz}
u_{\disc}(x,t) = \sum_{j\in\disc} y_j(t) \e^{\iu jx} \with \disc = \bigl\{\, -K,\dots,K-1 \,\bigr\}
\end{equation}
defined by its Fourier coefficients $y_j(t)$ with indices $j$ from the finite index set $\disc$ is used as an ansatz for the solution of the nonlinear wave equation. Inserting this ansatz in the nonlinear wave equation and evaluating in the collocation points $x_k=\pi k/K$ with $k\in\disc$ then leads to the system 
\begin{equation}\label{eq-nlw-semi}
\ddot{y}(t) = -\Omega^2 y(t) + f\klabig{y(t)}
\end{equation}
for the vector $y(t) = (y_j(t))_{j\in\disc}$ of Fourier coefficients 
(the vector $y(t)$ belongs to the set $\C^\disc$ of complex vectors indexed by $\disc$). 
Here, $\Om$ is a nonnegative and diagonal matrix containing frequencies $\om_j$,
\[
\Om = \diag(\om_j)_{j\in\disc} \with \om_j = \abs{j},
\]
and the nonlinearity $f$ is given by the discrete convolution $\conv$,
\begin{equation}\label{eq-f}
f(y) = \underbrace{y \conv \dotsm \conv y}_{\text{$p$ times}} \with \klabig{y \conv z}_j = \sum_{k+l \equiv j \bmod{2K}} y_k z_l, \qquad j\in\disc.
\end{equation}
The initial values $y(t_0)$ and $\dot{y}(t_0)$ for~\eqref{eq-nlw-semi} are determined from the initial values $u(\cdot,t_0)$ and $u_t(\cdot,t_0)$ of the nonlinear wave equation~\eqref{eq-nlw} by
\begin{equation}\label{eq-nlw-semi-init1}
y_j(t_0) = \sum_{k\in\Z : k\equiv j \bmod{2K}} u_k(t_0), \qquad \dot{y}_j(t_0) = \sum_{k\in\Z : k\equiv j \bmod{2K}} \dot{u}_{k}(t_0), \qquad j\in\disc,
\end{equation}
where we denote by $u_k(t)$ and $\dot{u}_{k}(t)$ the Fourier coefficients of $u(\cdot,t)$ and $u_t(\cdot,t)$, respectively. This choice is possible if these Fourier coefficients form absolutely summable sequences, and it
corresponds then to a trigonometric interpolation of $u(\cdot,t)$ and $u_t(\cdot,t)$ in the collocation points $x_k$, $k\in\disc$. If the initial values $u(\cdot,t)$ and $u_t(\cdot,t)$ are given by their Fourier coefficients, the choice
\begin{equation}\label{eq-nlw-semi-init2}
y_j(t_0) = u_j(t_0), \qquad \dot{y}_j(t_0) = \dot{u}_j(t_0), \qquad j\in\disc
\end{equation}
is computationally advantageous. An error analysis of the semi-discretization in space is given for both choices of initial values in \cite{Gauckler}. 

The exact solution of the spatially discrete system~\eqref{eq-nlw-semi} is given by the variation-of-constants formula
\begin{equation}\label{eq-voc}
\begin{pmatrix} y(t)\\ \dot{y}(t)\end{pmatrix} = R(t-t_0) \begin{pmatrix} y(t_0)\\ \dot{y}(t_0)\end{pmatrix} + \int_{t_0}^t R(t-\ta) \begin{pmatrix} 0\\ f(y(\ta))\end{pmatrix} \,\drm\ta
\end{equation}
with
\begin{equation}\label{eq-R}
R(t) = \begin{pmatrix} \cos(t\Om) & t \sinc(t\Om)\\ -\Om\sin(t\Om) & \cos(t\Om)\end{pmatrix}.
\end{equation}
Via~\eqref{eq-ansatz}, this solution $(y,\dot{y})$ gives an approximation $u_\disc(x,t)$ of the nonlinear wave equation~\eqref{eq-nlw}. For real-valued initial values $u(x,t_0)$ and $u_t(x,t_0)$, this approximation takes real values in the collocation points $x_k$. An approximation that takes real values in all $x\in\TT$ (and the same values in the collocation points) can be obtained by replacing $y_{-K}(t) \e^{\iu (-K)x}$ in the ansatz~\eqref{eq-ansatz} by $\frac12 y_{-K}(t) (\e^{\iu (-K)x} + \e^{\iu Kx})$.

\subsection{Trigonometric integrators for the discretization in time}\label{subsec-trigo}

For the discretization in time of the spatially discrete system~\eqref{eq-nlw-semi}, we consider \emph{trigonometric integrators} (or \emph{exponential integrators}) as described for instance in~\cite[Section~XIII.2.2]{Hairer2006}. We will restrict here to methods in a one-step formulation which can be considered as direct discretizations of the variation-of-constants formula~\eqref{eq-voc}. They compute approximations $y^n$ to $y(t_n)$ at discrete times $t_n=t_0+nh$ with the time step-size $h$ by
\begin{equation}\label{eq-trigo}
\begin{pmatrix} y^{n+1}\\ \dot{y}^{n+1}\end{pmatrix} = R(h) \begin{pmatrix} y^n\\ \dot{y}^n\end{pmatrix} + \begin{pmatrix} \sfrac12 h^2 \Psi f(\Phi y^n)\\ \sfrac12 h \Psi_0 f(\Phi y^n) + \sfrac12 h \Psi_1 f(\Phi y^{n+1})\end{pmatrix}.
\end{equation}
The diagonal matrices $\Phi$, $\Psi$, $\Psi_0$ and $\Psi_1$ are filters defined by
\[
\Phi = \phi(h\Om), \qquad \Psi = \psi(h\Om), \qquad \Psi_0 = \psi_0(h\Om), \qquad \Psi_1 = \psi_1(h\Om)
\]
with filter functions $\phi$, $\psi$, $\psi_0$ and $\psi_1$ that satisfy $\phi(0)=\psi(0)=\psi_0(0)=\psi_1(0)=1$.

The method~\eqref{eq-trigo} is determined by its filter functions $\psi$, $\phi$, $\psi_0$ and $\psi_1$. For even filter functions, it is symmetric if and only if
\begin{equation}\label{eq-symmetric}
\psi(\xi) = \sinc(\xi) \psi_1(\xi) \myand \psi_0(\xi) = \cos(\xi) \psi_1(\xi),
\end{equation}
and it is then symplectic if and only if
\begin{equation}\label{eq-symplectic}
\psi(\xi) = \sinc(\xi) \phi(\xi),
\end{equation}
see \cite[Section~XIII.2.2]{Hairer2006}.
Popular choices of the filter functions $\psi$, $\phi$, $\psi_0$ and $\psi_1$ are
\begin{align*}
&(\text{B})\qquad && \psi(\xi) = \sinc(\xi), && \phi(\xi) = 1, && \text{$\psi_0$ and $\psi_1$ as in~\eqref{eq-symmetric},}\\
&(\text{C})\qquad && \psi(\xi) = \sinc^2(\xi), && \phi(\xi) = \sinc(\xi), && \text{$\psi_0$ and $\psi_1$ as in~\eqref{eq-symmetric},}\\
&(\text{E})\qquad && \psi(\xi) = \sinc^2(\xi), && \phi(\xi) = 1, && \text{$\psi_0$ and $\psi_1$ as in~\eqref{eq-symmetric},}\\
&(\text{G})\qquad && \psi(\xi) = \sinc^3(\xi), && \phi(\xi) = \sinc(\xi), && \text{$\psi_0$ and $\psi_1$ as in~\eqref{eq-symmetric}.}\\
\intertext{The labels (B), (C), (E) and (G) of these methods are the ones used in \cite{Grimm2006,Hairer2006}. Method (B) goes back to Deuflhard \cite{Deuflhard1979} and coincides in our situation with the impulse method \cite{Grubmueller1991,Tuckerman1992}, and method (C) is the mollified impulse method proposed by Garc\'{i}a-Archilla, Sanz-Serna \& Skeel \cite{GarciaArchilla1999}. Method (E) was first considered by Hairer \& Lubich \cite{Hairer2000} and method (G) by Grimm \& Hochbruck \cite{Grimm2006}. 
\endgraf
Yet another method that we introduce here and which turns out to work well in the case of the wave equation is the method with filter functions}
&(\widetilde{\text{B}})\qquad && \psi(\xi) = \chi(\xi) \sinc(\xi), && \phi(\xi) = \chi(\xi), && \text{$\psi_0$ and $\psi_1$ as in~\eqref{eq-symmetric},}
\end{align*}
where $\chi=\chi_{[-\pi,\pi]}$ is the characteristic function of the interval $[-\pi,\pi]$. This method uses truncated versions of the filter functions of method (B). Similar modifications of methods (C), (E) and (G) are possible. Such methods are computationally attractive since they require only the evaluation of a reduced version of $f$ on reduced arguments.

We do not  consider here the method  of Gautschi \cite{Gautschi1961} and the Gautschi-type method of Hochbruck \& Lubich \cite{Hochbruck1999} (methods (A) and (D) in \cite{Grimm2006,Hairer2006}); in these symmetric two-step methods, one uses $\psi(\xi)=\sinc^2(\xi/2)$, and hence already the formulation as a one-step method~\eqref{eq-trigo} does not make sense because of singularities in  $\psi_0$ and $\psi_1$ defined by~\eqref{eq-symmetric}. 
As long as these singularities are avoided (by an appropriate choice of the time step-size $h$), the one-step formulation~\eqref{eq-trigo} does make sense and our theory of the following subsection applies equally to these methods.

\subsection{Error bounds}\label{subsec-result}

We collect all assumptions on the filter functions $\psi$, $\phi$, $\psi_0$ and $\psi_1$ defining the trigonometric method~\eqref{eq-trigo} that we will need in the sequel.

\begin{assum}\label{assum-filter}
For given $-1\le\be\le 1$, we assume that there exists a constant $c$ such that, for all $\xi=h\om_j$ with $j\in\disc$ and $\om_j\ne 0$,
\begin{subequations}\label{eq-bound-filter}
\begin{align}
&\abs{\phi(\xi)} \le c,\label{eq-bound-filter-phi}\\
&\abs{\psi(\xi)} \le c\, \xi^{\be} \myif -1\le\be\le 0,\label{eq-bound-filter-psi1}\\
&\abs{1-\psi(\xi)} \le c\, \xi^{\be} \myif 0<\be\le 1,\label{eq-bound-filter-psi2}\\
&\abs{1-\chi(\xi)} \le c\, \xi^{1+\be} \myfor \chi=\phi,\psi_0,\psi_1.\label{eq-bound-filter-chi}
\end{align}
\end{subequations}
\end{assum}

There are many methods that satisfy Assumption~\ref{assum-filter} uniformly, for all $-1\le\be\le 1$, for all step-sizes $h>0$ and for all spatial discretization parameters $K$. The methods (B), (C), (E), (G) and ($\widetilde{\text{B}}$) mentioned in the previous subsection all satisfy Assumption~\ref{assum-filter} with $c=2$ for all $-1\le\be\le 1$, all $h>0$ and all $K$. For a symmetric and symplectic method with even filter functions, that is a method satisfying~\eqref{eq-symmetric} and~\eqref{eq-symplectic}, the inequalities~\eqref{eq-bound-filter-phi}--\eqref{eq-bound-filter-chi} of the above assumption hold with $c= C+2$ if
\[
\abs{ \phi(\xi) } \le C, \qquad \abs{1-\phi(\xi)} \le C \xi^{1+\be}.
\]

Under Assumption~\ref{assum-filter}, we will prove in Section~\ref{sec-proof} the following main result on the error of the trigonometric integrator~\eqref{eq-trigo}. The error is measured, for $s\in\R$, in the norm
\[
\norm{y}_s = \biggl( \sum_{j\in\disc} \skla{j}^{2s} \abs{y_j}^2 \biggr)^{1/2} \with \skla{j} = \max\klabig{ 1,\abs{j} }
\]
for $y\in\C^\disc$.
This norm is (equivalent to) the Sobolev $H^s$-norm\footnote{For $s\notin\mathbb{N}$, the usual convention is used that $H^s$ is the Bessel potential space.} of the trigonometric polynomial $\sum_{j\in\disc} y_j \e^{\iu jx}$. For $y=y^n$, this trigonometric polynomial is the fully discrete approximation of the solution $u(\cdot,t_n)$ of the nonlinear wave equation.

\begin{theorem}\label{thm-main}
Let $c\ge 1$ and $s\ge 0$, and assume that the exact solution $(y(t),\dot{y}(t))$ of the spatial semi-discretization~\eqref{eq-nlw-semi} of the nonlinear wave equation~\eqref{eq-nlw} satisfies
\begin{equation}\label{eq-finiteenergy}
\norm{y(t)}_{s+1} + \norm{\dot{y}(t)}_s \le M \myfor 0\le t-t_0 \le T.
\end{equation}
Then, there exists $h_0>0$ such that for all time step-sizes $h\le h_0$ the following error bound holds for the numerical solution $(y^n,\dot{y}^n)$ computed with the trigonometric integrator~\eqref{eq-trigo}: If Assumption~\ref{assum-filter} holds with constant $c$ for $\be=0$ and $\be=\al$ with some $-1\le\al\le 1$, then
\[
\normbig{y(t_n) - y^n}_{s+1-\al} + \normbig{\dot{y}(t_n) - \dot{y}^n}_{s-\al} \le C h^{1+\al} \myfor 0\le t_n-t_0=nh \le T.
\]
The constants $C$ and $h_0$ depend only on $M$ and $s$ from~\eqref{eq-finiteenergy}, the power $p$ of the nonlinearity in~\eqref{eq-nlw}, the final time $T$ and the constant $c$.
\end{theorem}

The proof of the above theorem will be given in Section~\ref{sec-proof}. We emphasize that the error bounds are uniform in the spatial discretization parameter $K$. 
They can be combined with the error analysis of the semi-discretization in space as given in \cite{Gauckler} to yield error bounds for the full discretization.

For $s=0$, the assumption $\norm{y(t)}_{s+1} + \norm{\dot{y}(t)}_s \le M$ in Theorem~\ref{thm-main} is basically a finite energy assumption on the solution of the nonlinear wave equation~\eqref{eq-nlw} and its spatial semi-discretization~\eqref{eq-nlw-semi}. In this case, Theorem~\ref{thm-main} yields, for example, a second-order error bound for $y$ in $L^2$ ($\al=1$), and a first-order error bound for $\dot{y}$ in~$L^2$ ($\al=0$).

To obtain second-order error bounds for $y$ and first-order error bounds for $\dot{y}$, 
similar but stronger assumptions on the filter functions have been used in \cite[Equations~(11)--(16)]{Grimm2006} and \cite[Equation (4.1) of Section XIII.4]{Hairer2006} to treat a slightly different kind of second-order oscillatory differential equations (note that in \cite[Section XIII.4]{Hairer2006} the regime $h\om_j\ge c_0>0$ for $j\ne 0$ is considered, in which~\eqref{eq-bound-filter-psi2} is implied by~\eqref{eq-bound-filter-psi1}, for instance). The methods (B) and (E), which do not use a filter inside the nonlinearity, do not satisfy the assumptions of \cite{Grimm2006,Hairer2006} for all step-sizes $h>0$ and do not show second-order convergence for the equations considered therein. Our error analysis covers these methods and hence shows that, in the case of the nonlinear wave equation, filtering inside the nonlinearity is indeed not necessary, at least when it comes to error bounds on bounded time intervals, see also Subsection~\ref{subsec-useoffilter}. This has also been observed by Cano \& Moreta \cite{Cano2010a}. 

Our assumptions~\eqref{eq-bound-filter} on the filter functions are not fulfilled for the method of Gautschi and the Gautschi-type method of Hochbruck \& Lubich (methods (A) and (D) in \cite{Grimm2006,Hairer2006}) whenever the product $h\om_j$ of the time step-size $h$ and a frequency $\om_j$ is close to an odd integer multiple of $\pi$ for some $j\in\disc$. An error analysis of the Gautschi-type method of Hochbruck \& Lubich when applied to the Sine--Gordon equation is given in~\cite{Grimm2006a}. With a combination of the proof as given there and the proof to be presented in the present paper, it should be possible to extend this analysis to nonlinear wave equations~\eqref{eq-nlw} with polynomial nonlinearities.

\subsection{Numerical experiments}\label{subsec-numexp}

We illustrate the error bounds of Theorem~\ref{thm-main} by numerical experiments\footnote{The numerical experiments used an implementation of a Pad\'{e} approximation of the function $\sinc$ that was kindly provided by Georg Jansing (Universit\"at D\"usseldorf).} for the quadratic nonlinear wave equation, that is~\eqref{eq-nlw} with $p=2$. We consider the spatial semi-discretization~\eqref{eq-nlw-semi} for several values of the spatial discretization parameter $K$,
\[
K=2^5, \, 2^7, \, 2^9, \, 2^{11}, \, 2^{13},
\]
and approximate it with the trigonometric integrator~\eqref{eq-trigo}. 

\begin{figure}[t]
\centering
\includegraphics{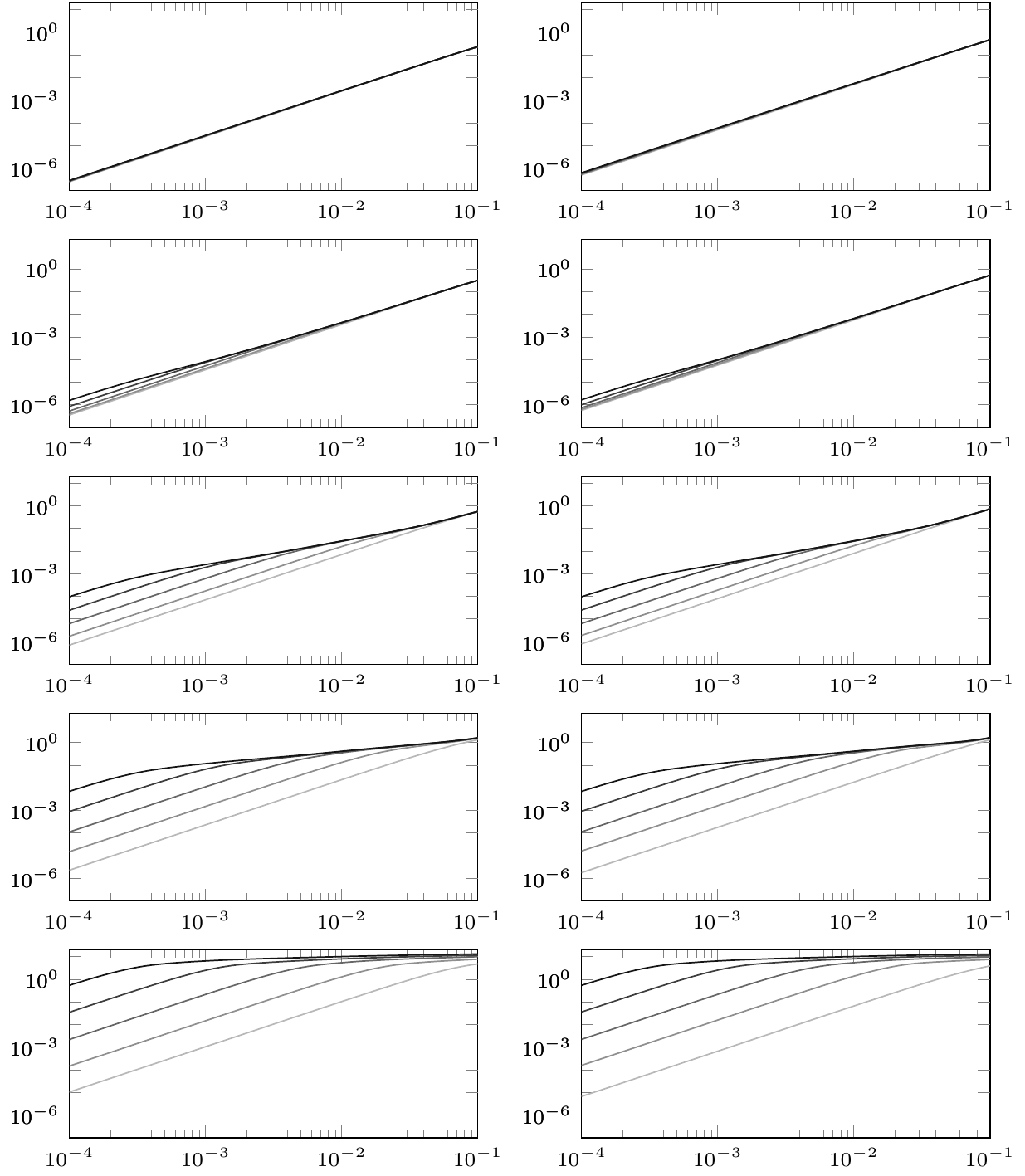}
\caption{Mollified impulse method (C): Errors $\norm{y(t_n)-y^{n}}_{1-\al}$ (left column) and $\norm{\dot{y}(t_n)-\dot{y}^{n}}_{-\al}$ (right column) at time $t_n=t_0+1$ versus the time step-size $h$ with $\al=1,\frac12,0,-\frac12,-1$ (from top to bottom). Different grey tones correspond to different values of the spatial discretization parameter~$K$.}\label{fig-mollified}
\end{figure}

As initial value for~\eqref{eq-nlw-semi} we choose vectors
\[
\klabig{y(t_0),\dot{y}(t_0)}\in\C^\disc\times\C^\disc
\]
that are bounded uniformly in the spatial discretization parameter $K$
\[
\text{in} \qquad H^{s+1}\times H^s \qquad \text{for} \quad s=0 \quad \text{but not for} \quad s\ge \sfrac1{100}.
\]
More precisely, we choose coefficients $y_j(t_0)$ and $\dot{y}_j(t_0)$ on the complex unit circle and then scale them by $\skla{j}^{-1.51}$ and $\skla{j}^{-0.51}$, respectively. The choice of complex numbers on the unit circle is more or less randomly; we only ensure that the corresponding trigonometric polynomial takes real values in the collocation points. In this way, we get initial values $y(t_0)$ and $\dot{y}(t_0)$ that satisfy the condition~\eqref{eq-finiteenergy} of Theorem~\ref{thm-main} at time $t=t_0$ uniformly in $K$ for $s=0$ but not for $s\ge \frac1{100}$, and we expect that this holds true on a finite time interval. 

For the discretization in time, we first use the mollified impulse method (method (C) of Subsection~\ref{subsec-trigo}). In Figure~\ref{fig-mollified}, we plot the errors $y(t_n)-y^{n}$ (left column) and $\dot{y}(t_n)-\dot{y}^{n}$ (right column) at time $t_n=t_0+1$ in dependence of the time step-size $h$. In the different rows of Figure~\ref{fig-mollified}, these errors are measured in different Sobolev norms: we plot
\[
\normbig{y(t_n)-y^{n}}_{1-\al} \myand \normbig{\dot{y}(t_n)-\dot{y}^{n}}_{-\al}
\]
as functions of $h$ with, from top to bottom,
\[
\al=1,\,\sfrac12,\,0,\,-\sfrac12,\,-1.
\]
In different grey tones, we plot the results for different values of the spatial discretization parameter $K$. Being interested in the order of convergence that is uniform in $K$, we clearly observe a dependence of this order on the considered norm. The observed order of convergence that is uniform in $K$ is $1+\al$, in agreement with Theorem~\ref{thm-main}. This illustrates the sharpness of the error bounds of this theorem with respect to both, the order of convergence and the considered Sobolev space. We finally observe that, under the CFL-type step-size restriction $hK\le \pi$, 
the convergence is of order two in all norms. The figures clearly show that this second-order convergence is not uniform in $K$ for $\al<1$.

\begin{figure}[t]
\centering
\includegraphics{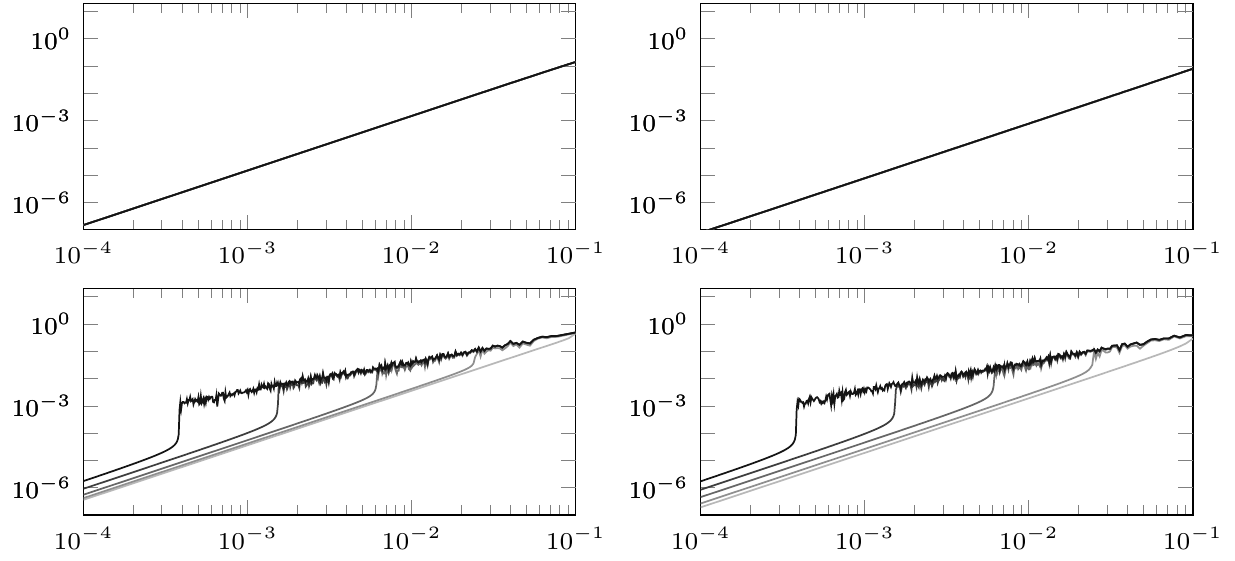}
\caption{Deuflhard/impulse method (B): Errors $\norm{y(t_n)-y^{n}}_{1-\al}$ (left column) and $\norm{\dot{y}(t_n)-\dot{y}^{n}}_{-\al}$ (right column) at time $t_n=t_0+1$ versus the time step-size $h$ with $\al=\frac12,-\frac12$ (from top to bottom). Different grey tones correspond to different values of the spatial discretization parameter~$K$.}\label{fig-deuflhard}
\end{figure}

If method (B) of Subsection~\ref{subsec-trigo} (the method of Deuflhard which coincides with the impulse method) 
is used instead of method (C), we observe a slightly different behaviour. In Figure~\ref{fig-deuflhard}, the errors $y(t_n)-y^{n}$ (left column) and $\dot{y}(t_n)-\dot{y}^{n}$ (right column) at time $t_n=t_0+1$ of this method are plotted. We observe second-order convergence of $(y,\dot{y})$ uniformly in $K$ not only in $H^0\times H^{-1}$, as suggested by Theorem~\ref{thm-main}, but also in $H^{1/2}\times H^{-1/2}$. First-order convergence uniformly in $K$ is observed in $H^{3/2}\times H^{1/2}$, instead of $H^{1}\times H^{0}$ as for the mollified impulse method (C). At present, we do not have a theoretical explanation for this improved convergence behaviour of method~(B). 

\begin{figure}[t]
\centering
\includegraphics{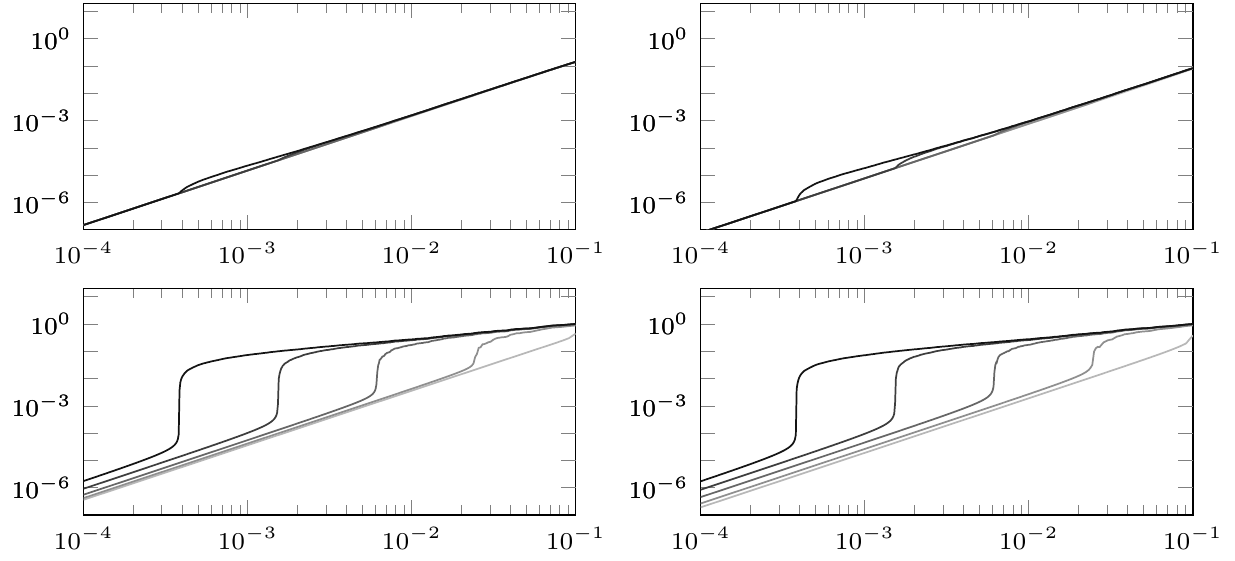}
\caption{Method ($\widetilde{\text{B}}$): Errors $\norm{y(t_n)-y^{n}}_{1-\al}$ (left column) and $\norm{\dot{y}(t_n)-\dot{y}^{n}}_{-\al}$ (right column) at time $t_n=t_0+1$ versus the time step-size $h$ with $\al=\frac12,-\frac12$ (from top to bottom). Different grey tones correspond to different values of the spatial discretization parameter~$K$.}\label{fig-newdeufl}
\end{figure}

This exceptionally good behaviour of a trigonometric integrator seems to be restricted to this particular method. 
For methods (E) and (G) of Subsection~\ref{subsec-trigo}, the results are qualitatively the same as for method (C) in Figure~\ref{fig-mollified}, and this behaviour can again be completely explained with Theorem~\ref{thm-main}. For method ($\widetilde{\text{B}}$) of Subsection~\ref{subsec-trigo}, the results are qualitatively slightly different from those for method (C), see Figure~\ref{fig-newdeufl}, but they still can be completely explained with Theorem~\ref{thm-main}.

\begin{figure}[t]
\centering
\includegraphics{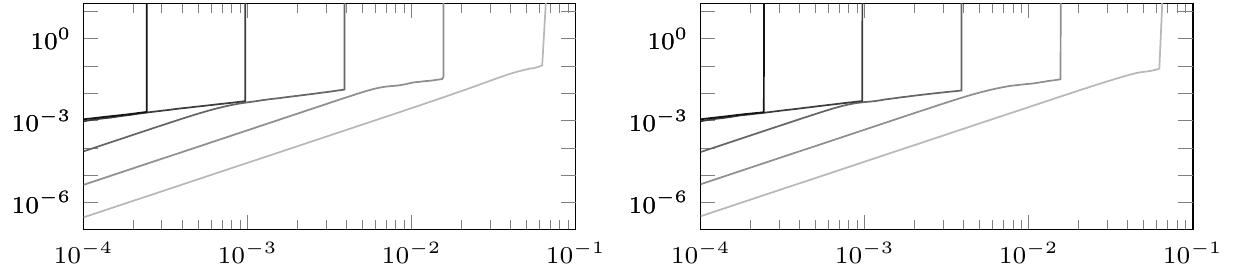}
\caption{St\"ormer--Verlet/leapfrog method: Errors $\norm{y(t_n)-y^{n}}_{1-\al}$ (left column) and $\norm{\dot{y}(t_n)-\dot{y}^{n}}_{-\al}$ (right column) at time $t_n=t_0+1$ versus the time step-size $h$ with $\al=1$. Different grey tones correspond to different values of the spatial discretization parameter~$K$.}\label{fig-stoermerverlet}
\end{figure}

It is interesting to compare the observed and theoretically explained convergence behaviour of trigonometric integrators with the behaviour of the St\"ormer--Verlet/leap\-frog discretization in time, one of the widely used discretizations of wave equations. See Subsection~\ref{subsec-sv} below for a description of the method when applied to systems of the form~\eqref{eq-nlw-semi}. Repeating the experiment described above with the St\"ormer--Verlet/leapfrog method gives Figure~\ref{fig-stoermerverlet}. The well-known instability of this method if $h\om_j> 2$ for some $j\in\disc$, i.e., $hK> 2$, is clearly visible. Under the step-size restriction $hK\le 2$, we observe in addition that the uniform convergence of the St\"ormer--Verlet/leapfrog method in $H^0\times H^{-1}$ is of order $2/3$. In comparison, the considered trigonometric integrators are in $H^0\times H^{-1}$ second-order convergent uniformly in $K$, even without the step-size restriction $hK\le 2$, see Figures~\ref{fig-mollified}--\ref{fig-newdeufl} and Theorem~\ref{thm-main}. An explanation of the convergence behaviour of the St\"ormer--Verlet/leapfrog discretization will be given in Subsection~\ref{subsec-sv}.

\section{Proofs of the error bounds of Theorem~\ref{thm-main}}\label{sec-proof}

\subsection{Estimates of the nonlinearity}\label{subsec-nonlinearity}

We prove some important, yet elementary, estimates of the nonlinearity $f$ in the spatial semi-discretization~\eqref{eq-nlw-semi}. These estimates give us a tool to climb the scale of Sobolev up and down, but on the other hand, they also force us to do so. We emphasize that all estimates given in this section are uniform in the spatial discretization parameter $K$ from Subsection~\ref{subsec-collocation}.

We begin with the following estimates of the convolution of two vectors in the spaces $H^{\si}$, $\si\in\R$. At least some special cases of these estimates are known, see \cite[Lemma~4.2]{Hairer2008}. 

\begin{proposition}\label{prop-algebra}
(i) Let $\si,\si'\in\R$ with $\si'\ge\abs{\si}$ and $\si'\ge 1$. 
We then have, for $y,z\in\C^{\disc}$,
\[
\norm{y\conv z}_{\si} \le C \norm{y}_{\si'} \norm{z}_{\si}
\]
with a constant $C$ depending only on $\abs{\si}$.

(ii) Let $\si\in\R$ with $\si\ge -1$. 
We then have, for $y,z\in\C^{\disc}$,
\[
\norm{y\conv z}_{\si} \le C \norm{y}_{\si+1} \norm{z}_{\si+1}
\]
with a constant $C$ depending only on $\abs{\si}$.
\end{proposition}
\begin{proof}
We first show, for $s,s',s''\in\R$ with $0\le s\le s'$, $0\le s\le s''$ and $s'+s''-s\ge 1$, the inequality 
\begin{equation}\label{eq-inequ-fund}
\sum_{k\in\disc} \frac{\skla{j}^{2s}}{\skla{k}^{2s'} \skla{j- k \bmod{2K}}^{2s''} } \le C
\end{equation}
with a generic constant $C$ depending on $s$ but not on $j\in\disc$. 
Here, we denote by $j- k \bmod{2K}$ the index in the finite set $\disc$ that is congruent to $j-k$ modulo $2K$.
Using
\[
\skla{j} \le \skla{k+(j- k \bmod{2K})} \le \skla{k} + \skla{j- k \bmod{2K}} \le 2\max\klabig{ \skla{k}, \skla{j- k \bmod{2K}} }
\]
for $j\in\disc$ together with $0\le s\le s'$ and $0\le s\le s''$ shows that 
\[
\frac{\skla{j}^{2 s}}{\skla{k}^{2 s'} \skla{j- k \bmod{2K}}^{2 s''} } \le \frac{2^{2 s}}{\min\klabig{ \skla{k}, \skla{j- k \bmod{2K}} }^{2(s'+s''-s)} }.
\]
With $s'+s''-s\ge 1$ and $1/\min(a,b)^2 \le 1/a^2+1/b^2$ for $a,b>0$ we thus see that the sum in~\eqref{eq-inequ-fund} is dominated by the convergent sum $2^{2 s+1} \sum_{k\in\Z} \skla{k}^{-2} = 2^{2 s+1}(1+\pi^2/3)$.

With the help of the inequality~\eqref{eq-inequ-fund}, we now prove statements (i) and (ii) of the proposition. We distinguish between $\si\ge 0$ and $\si\le 0$.

(a) First, we consider the case $\si\ge 0$. Let $s,s',s''\in\R$ to be chosen later. We have
\begin{equation}\label{eq-prop-algebra-aux}
\norm{y\conv z}_{s}^2 = \sum_{j\in\disc} \skla{j}^{2s} \absbigg{\sum_{k\in\disc} y_k z_{j-k\bmod{2K}}}^2. 
\end{equation}
Applying the Cauchy-Schwarz inequality to the second sum yields
\begin{align*}
\norm{y\conv z}_{s}^2 \le \sum_{j\in\disc} &\klabigg{\sum_{k\in\disc} \frac{\skla{j}^{2s}}{\skla{k}^{2s'} \skla{j-k\bmod{2K}}^{2s''}} }\\
 \cdot &\klabigg{ \sum_{k\in\disc} \skla{k}^{2s'} \abs{y_k}^2 \skla{j-k\bmod{2K}}^{2s''} \abs{z_{j-k\bmod{2K}}}^2 }.
\end{align*}
Choosing $s=s''=\si$ and $s'=\si'$ and using the inequality~\eqref{eq-inequ-fund} then shows the statement (i) of the proposition. Similarly, statement (ii) follows from~\eqref{eq-inequ-fund} with $s=\si$ and $s'=s''=\si+1$.

(b) Finally, we consider the case $\si\le 0$. Let again $s,s',s''\in\R$. Applying the Cauchy-Schwarz inequality to the second sum of~\eqref{eq-prop-algebra-aux}, but in a different way than in step (a), yields
\[
\norm{y\conv z}_{-s'}^2 \le \sum_{j\in\disc} \skla{j}^{-2s'} \klabigg{\sum_{k\in\disc} \skla{k}^{2s''} \abs{y_k}^2 } \klabigg{ \sum_{k\in\disc} \frac{1}{\skla{j-k \bmod{2K}}^{2s''}} \abs{z_{k}}^2 },
\]
and hence
\[
\norm{y\conv z}_{-s'}^2 \le \norm{y}_{s''}^2 \sum_{k\in\disc} \klabigg{ \sum_{j\in\disc}\frac{\skla{k}^{2s}}{\skla{j}^{2s'} \skla{j-k\bmod{2K}}^{2s''}}} \skla{k}^{-2s} \abs{z_{k}}^2.
\]
Statement (i) now follows from~\eqref{eq-inequ-fund} with $s=s'=-\si$ and $s''=\si'$. For statement (ii) we use $s=0$, $s'=-\si$ and $s''=\si+1$, and then $\norm{z}_0\le\norm{z}_{\si+1}$.
\end{proof}

These estimates of the convolution allow us to prove the following important properties of the nonlinearity $f(y)$ given by~\eqref{eq-f}.

\begin{proposition}\label{prop-nonlinearity}
Let $\si,\si'\in\R$ with $\si'\ge\abs{\si}$ and $\si'\ge 1$. If 
\[
\norm{y}_{\si'} \le M,\qquad \norm{z}_{\si'} \le M,
\]
then 
\begin{subequations}
\begin{align}
\norm{ f(y) - f(z) }_{\si} &\le C \norm{y-z}_{\si}, \label{eq-lipschitz}\\
\norm{ f(y) }_{\si'} &\le C \label{eq-nonlinearity}
\end{align}
\end{subequations}
with a constant $C$ depending on $M$, $\abs{\si}$, $\si'$ and $p$.
\end{proposition}
\begin{proof}
The estimate~\eqref{eq-nonlinearity} follows from Proposition~\ref{prop-algebra} (i) applied $p-1$ times with $\si'=\si$. Also the estimate~\eqref{eq-lipschitz} follows from part (i) of this proposition applied $p-1$ times to
\[
f(y)-f(z)= \sum_{j=0}^{p-1} 
\underbrace{y\conv \dotsm \conv y}_{\text{$j$ times}} \conv \underbrace{z\conv \dotsm \conv z}_{\text{$p-j-1$ times}\hspace{-12pt}} \, \conv \, (y-z) .\qedhere
\]
\end{proof}

\begin{proposition}\label{prop-nonlinearity2}
Let $s\ge 0$. If, for $y\colon[t_0,t_1]\rightarrow\C^\disc$,
\[
\normbig{y(t)}_{s+1} \le M,\qquad \normbig{\dot{y}(t)}_{s} \le M \myfor t_0\le t\le t_1,
\]
then 
\begin{subequations}
\begin{equation}\label{eq-deriv1-nonlinearity}
\normbigg{ \frac{\drm}{\drm t} f\klabig{y(t)} }_{s} \le C
\end{equation}
with a constant $C$ depending on $M$, $s$ and $p$.
If, in addition,
\[
\normbig{\ddot{y}(t)}_{s-1} \le M \myfor t_0\le t\le t_1,
\]
then 
\begin{equation}\label{eq-deriv2-nonlinearity}
\normbigg{ \frac{\drm^2}{\drm t^2} f\klabig{y(t)} }_{s-1} \le C
\end{equation}
with a constant $C$ depending on $M$, $s$ and $p$.
\end{subequations}
\end{proposition}
\begin{proof}
The first estimate~\eqref{eq-deriv1-nonlinearity} follows from Proposition~\ref{prop-algebra} (i) applied $p-1$ times with $\si'=s+1$ and $\si=s$ to
\[
\frac{\drm}{\drm t}f(y(t)) = p\, \underbrace{y(t)\conv\dotsm\conv y(t)}_{\text{$p-1$ times}} \,\conv\, \dot{y}(t) . 
\]
The second estimate~\eqref{eq-deriv2-nonlinearity} follows from Proposition~\ref{prop-algebra} (i) applied with $\si'=s+1$ and $\si=s-1$ and from Proposition~\ref{prop-algebra} (ii) applied with $\si=s-1$ to
\[
\frac{\drm^2}{\drm t^2}f(y(t)) = p \, \underbrace{y(t)\conv\dotsm\conv y(t)}_{\text{$p-1$ times}} \,\conv\, \ddot{y}(t) + p(p-1) \, \underbrace{y(t)\conv\dotsm\conv y(t)}_{\text{$p-2$ times}} \,\conv\, \dot{y}(t)\conv \dot{y}(t)
\]
in a similar way as in the proof of the first estimate~\eqref{eq-deriv1-nonlinearity}.
\end{proof}

\subsection{Proof of the lower order error bounds in higher order Sobolev spaces}\label{subsec-proof1}

We give the proof of Theorem~\ref{thm-main} for $-1\le \al\le 0$, assuming throughout that $h\le 1$. The proof follows the classical scheme of Lady Windermere's fan based on  a local error bound in Proposition~\ref{prop-local-uncond} below and a stability estimate in Proposition~\ref{prop-stab-uncond}. 

We will make use of the norm
\[
\normv{(y,\dot{y})}_{\si} = \bigl(\norm{y}_{\si+1}^2 + \norm{\dot{y}}_{\si}^2 \bigr)^{1/2},
\]
on $H^{\si+1}\times H^\si$ for various values of $\si\in\R$. We denote throughout by $(y(t),\dot{y}(t))$ the solution~\eqref{eq-voc} of the system~\eqref{eq-nlw-semi} and by $(y^0,\dot{y}^0),(y^1,\dot{y}^1),\ldots$ its numerical approximation~\eqref{eq-trigo}. 

Before studying local error and stability of the numerical method~\eqref{eq-trigo}, we prove the following lemma on the preservation of regularity of the numerical solution over one time step.

\begin{lemma}\label{lemma-reg}
Let $s\ge 0$ and $-1\le \al\le 0$, and assume that the filter functions satisfy Assumption~\ref{assum-filter} for $\be=\al$ with constant $c$. If 
\[
\normvbig{\klabig{y^0,\dot{y}^0}}_{s} \le M,
\]
then
\[
\norm{y^1}_{s+1} \le C
\]
with a constant $C$ depending on $M$, $s$, $p$ and $c$.
\end{lemma}
\begin{proof}
We have, by the definition of the method~\eqref{eq-trigo},
\[
\norm{y^1}_{s+1} \le \normbig{ \cos(h\Om) y^0 }_{s+1} + h \normbig{ \sinc(h\Om) \dot{y}^0}_{s+1} + \sfrac12 h^2 \normbig{ \Psi f( \Phi y^0)}_{s+1}.
\]
We then use $\sinc(0)\le h^{-1}$, the bound $\abs{\sinc(\xi)}\le \xi^{-1}$ for $\xi>0$, the bound $\abs{\psi(\xi)} \le c \xi^\al$ for $\xi=h\om_j>0$ of~\eqref{eq-bound-filter-psi1} and $\psi(0)\le h^{\al}$ to get
\[
\norm{y^1}_{s+1} \le \norm{y^0}_{s+1} + \norm{ \dot{y}^0}_s + \sfrac12 c h^{2+\al} \normbig{ f( \Phi y^0)}_{s+1+\al}.
\]
The fact that $-1\le \al\le 0$, the bound \eqref{eq-bound-filter-phi} of $\phi$ and the estimate~\eqref{eq-nonlinearity} from Proposition~\ref{prop-algebra} with $\si'=s+1$ then imply the stated bound of $y^1$ in $H^{s+1}$.
\end{proof}

Now, we study the local error of the trigonometric integrator~\eqref{eq-trigo}.

\begin{proposition}[Local error in $H^{s+1-\al}\times H^{s-\al}$ for $-1\le\al\le 0$]\label{prop-local-uncond}
Let $s\ge 0$ and $-1\le\al\le 0$, and assume that the filter functions satisfy Assumption~\ref{assum-filter} for $\be=\al$ with constant $c$. If 
\[
\normvbig{\klabig{y(\ta),\dot{y}(\ta)}}_{s} \le M \myfor t_0\le \ta\le t_1,
\]
then
\[
\normvbig{\klabig{y(t_1),\dot{y}(t_1)} - \klabig{y^1,\dot{y}^1}}_{s-\al} \le C h^{2+\al}
\]
with a constant $C$ depending on $M$, $s$, $p$ and $c$.
\end{proposition}
\begin{proof}
Throughout the proof, we denote by $C$ a generic constant depending on $M$, $s$, $p$ and $c$. 

(a) The local error $y(t_1)-y^1$ is of the form
\begin{equation}\label{eq-local-error-y}
y(t_1)-y^1 = \int_{t_0}^{t_1} (t_1-\ta) \sinc\klabig{(t_1-\ta)\Om} f\klabig{y(\ta)} \,\drm\ta - \sfrac12 h^2 \Psi f\klabig{\Phi y(t_0)},
\end{equation}
see~\eqref{eq-voc} and~\eqref{eq-trigo}.
We estimate both terms on the right-hand side separately. Similarly as in the proof of Lemma~\ref{lemma-reg}, we use that $h^\al\ge 1$, that $\abs{\sinc(\xi)}\le \xi^{\al}$ for $\xi>0$ and that $\abs{\psi(\xi)}\le c\xi^\al$ for $\xi=h\om_j>0$ by~\eqref{eq-bound-filter-psi1} to get
\[
\normbig{y(t_1)-y^1}_{s+1-\al} \le h^{2+\al} \sup_{t_0\le\ta\le t_1} \normbig{f\klabig{y(\ta)}}_{s+1} + \sfrac12 c h^{2+\al} \normbig{f\klabig{\Phi y(t_0)}}_{s+1}.
\]
Together with~\eqref{eq-nonlinearity} from Proposition~\ref{prop-nonlinearity} with $\si'=s+1$ and the bound~\eqref{eq-bound-filter-phi} of $\Phi$, this yields
\begin{equation}\label{eq-prop-localerror-aux1}
\normbig{y(t_1)-y^1}_{s+1-\al} \le C h^{2+\al}.
\end{equation}

(b) The local error $\dot{y}(t_1)-\dot{y}^1$ is of the form
\[
\dot{y}(t_1)-\dot{y}^1 = \int_{t_0}^{t_1} \cos\klabig{(t_1-\ta)\Om} f\klabig{y(\ta)} \,\drm\ta - \sfrac12 h \Psi_0 f\klabig{\Phi y(t_0)} - \sfrac12 h \Psi_1 f\klabig{\Phi y^1}.
\]
We split it as follows:
\begin{subequations}\label{eq-local-error-doty-split}
\begin{align}
\dot{y}(t_1)-\dot{y}^1 &= \int_{t_0}^{t_1} \klaBig{ \cos\klabig{(t_1-\ta)\Om} - \Id} f\klabig{y(\ta)} \,\drm\ta \label{eq-local-error-doty-1}\\
 &\qquad + \int_{t_0}^{t_1} f\klabig{y(\ta)} \,\drm\ta - \sfrac12 h \klaBig{ f\klabig{y(t_0)} + f\klabig{y(t_1)} } \label{eq-local-error-doty-2}\\
 &\qquad + \sfrac12 h \klaBig{ f\klabig{y(t_0)} - f\klabig{\Phi y(t_0)} } + \sfrac12 h \klaBig{ f\klabig{y(t_1)} - f\klabig{\Phi y^1} } \label{eq-local-error-doty-3}\\
 &\qquad + \sfrac12 h \klabig{\Id-\Psi_0} f\klabig{ \Phi y(t_0)} + \sfrac12 h \klabig{\Id-\Psi_1} f\klabig{ \Phi y^1}. \label{eq-local-error-doty-4}
\end{align}
\end{subequations}
We then use $\abs{\cos(\xi)-1} = 2 \abs{\sin(\xi/2)}^2 \le 2^{-\al} \xi^{1+\al}$ and~\eqref{eq-nonlinearity} from Proposition~\ref{prop-nonlinearity} with $\si'=s+1$ to estimate the term on right-hand side of~\eqref{eq-local-error-doty-1}:
\[
\normbig{\text{term on right-hand side of~\eqref{eq-local-error-doty-1}}}_{s-\al} \le C h^{2+\al}.
\]
The second component~\eqref{eq-local-error-doty-2} of the local error $\dot{y}(t_1)-\dot{y}^1$ is estimated at first as follows:
\[
\normbig{\text{term~\eqref{eq-local-error-doty-2}}}_{s-\al} \le h^{1+\al} \normbig{\text{term~\eqref{eq-local-error-doty-2}}}_{s+1} + h^{\al} \normbig{\text{term~\eqref{eq-local-error-doty-2}}}_{s},
\]
since $1\le \xi^{1+\al} + \xi^{\al}$ for $\xi>0$. An application of~\eqref{eq-nonlinearity} from Proposition~\ref{prop-nonlinearity} with $\si'=s+1$ to all terms of~\eqref{eq-local-error-doty-2} yields an estimate $Ch$ in the norm $\norm{\cdot}_{s+1}$. For an estimate in the norm $\norm{\cdot}_{s}$, we note that~\eqref{eq-local-error-doty-2} is the quadrature error of the trapezoidal rule. With its first-order Peano kernel $K_1(\si) = \frac12 -\si$ we thus get 
\[
\normbig{\text{term~\eqref{eq-local-error-doty-2}}}_s = h^2 \normbigg{\int_0^1 K_1(\si) \frac{\drm }{\drm t} f\klabig{y(t_0+\si h)} \, \drm\si }_s \le Ch^2,
\]
where we have used~\eqref{eq-deriv1-nonlinearity} from Proposition~\ref{prop-nonlinearity2} in the last estimate. In summary, we thus have
\[
\normbig{\text{term~\eqref{eq-local-error-doty-2}}}_{s-\al} \le C h^{2+\al}. 
\]
For the third term~\eqref{eq-local-error-doty-3} we use Lemma~\ref{lemma-reg}, the bound~\eqref{eq-bound-filter-phi} of $\phi$ and the estimate~\eqref{eq-lipschitz} from Proposition~\ref{prop-nonlinearity} with $\si=s-\al$ and $\si'=s+1$. This yields
\[
\normbig{\text{term~\eqref{eq-local-error-doty-3}}}_{s-\al} \le Ch \Bigl( \normbig{ (\Id-\Phi) y(t_0) }_{s-\al} + \normbig{y(t_1)-y^1}_{s-\al} + \normbig{(\Id-\Phi) y^1}_{s-\al} \Bigr),
\]
where we have split in addition $y(t_1)-\Phi y^1= (y(t_1)-y^1)+(y^1-\Phi y^1)$.
We then use the bound~\eqref{eq-bound-filter-chi} of $1-\phi$, the above local error bound~\eqref{eq-prop-localerror-aux1} of $y(t_1)-y^1$ (note that $\norm{z}_{s-\al}\le \norm{z}_{s+1-\al}$ for $z\in\C^\disc$) and Lemma~\ref{lemma-reg} to get 
\[
\normbig{\text{term~\eqref{eq-local-error-doty-3}}}_{s-\al} \le Ch^{2+\al}.
\]
For the last term~\eqref{eq-local-error-doty-4} we similarly use the bounds~\eqref{eq-bound-filter-chi} of $1-\psi_0$ and $1-\psi_1$, the bound~\eqref{eq-bound-filter-phi} of $\phi$, Lemma~\ref{lemma-reg} and~\eqref{eq-nonlinearity} from Proposition~\ref{prop-nonlinearity} with $\si'=s+1$ to get
\[
\normbig{\text{term~\eqref{eq-local-error-doty-4}}}_{s-\al} \le Ch^{2+\al}.
\]
Putting all these estimates of the single terms in~\eqref{eq-local-error-doty-split} together yields the claimed local error bound of order $2+\al$ for $\norm{\dot{y}(t_1) - \dot{y}^1}_{s-\al}$.
\end{proof}

\begin{proposition}[Stability in $H^{s+1-\al}\times H^{s-\al}$ for $-1\le\al\le0$]\label{prop-stab-uncond}
Let $s\ge 0$ and $-1\le\al\le0$, and assume that the filter functions satisfy Assumption~\ref{assum-filter} for $\be=\al$ with constant $c$. We consider the trigonometric integrator~\eqref{eq-trigo} with different initial values $(y^0,\dot{y}^0)$ and $(z^0,\dot{z}^0)$. If 
\[
\normvbig{\klabig{y^0,\dot{y}^0}}_{s}\le M \myand \normvbig{\klabig{z^0,\dot{z}^0}}_{s}\le M,
\]
then 
\[
\normvbig{\klabig{y^{1},\dot{y}^1} - \klabig{z^{1},\dot{z}^1}}_{s-\al} \le  \bigl( 1 + C h \bigr) \normvbig{\bigl(y^{0},\dot{y}^0\bigr) - \bigl(z^{0},\dot{z}^0\bigr)}_{s-\al}
\]
with a constant $C$ depending on $M$, $s$, $p$ and $c$. 
\end{proposition}
\begin{proof}
We first study the behaviour of $R(h)$ under the norm $\normv{\cdot}_{\si}$. For
\[
\begin{pmatrix} w \\ \dot{w} \end{pmatrix}
 = R(h)
\begin{pmatrix} v\\ \dot{v} \end{pmatrix},
\]
we have
\begin{equation}\label{eq-normR}
\normv{(w,\dot{w})}_{\si} = \normvbig{(v,\dot{v}) + h(\widetilde{v},0)}_{\si}, 
\end{equation}
where $\widetilde{v}\in\C^\disc$ is zero except in its component with index $0$ in which it takes the value $\dot{v}_0$, i.e., $\widetilde{v}_j=\delta_{j,0} \dot{v}_j$ with the Kronecker delta. 
This shows that
\begin{subequations}\label{eq-stab}
\begin{align}
\normvbig{\bigl(y^1,\dot{y}^1\bigr) - \bigl(z^1,\dot{z}^1\bigr)}_{s-\al} &\le \normvbig{\bigl(y^0,\dot{y}^0\bigr) - \bigl(z^0,\dot{z}^0\bigr)}_{s-\al}\notag\\
 &\qquad + h \absbig{ \dot{y}_0^0 - \dot{z}_0^0 }\label{eq-stab0}\\
 &\qquad + \sfrac12 h^2 \normbig{\Psi \bigl(f(\Phi y^0) - f(\Phi z^0)\bigr)}_{s+1-\al}\label{eq-stab1}\\
 &\qquad + \sfrac12 h \normbig{\Psi_0 \bigl(f(\Phi y^0) - f(\Phi z^0)\bigr)}_{s-\al}\label{eq-stab2}\\
 &\qquad + \sfrac12 h \normbig{\Psi_1 \bigl(f(\Phi y^1) - f(\Phi z^1)\bigr)}_{s-\al}\label{eq-stab3}.
\end{align}
\end{subequations}
We estimate the terms~\eqref{eq-stab0}--\eqref{eq-stab3} separately. We have
\[
\text{term~\eqref{eq-stab0}} \le h \normbig{\dot{y}^0-\dot{z}^0}_{s-\al}
\]
Using the bound~\eqref{eq-bound-filter-psi1} of $\psi$, the estimate~\eqref{eq-lipschitz} from Proposition~\ref{prop-nonlinearity} with $\si=\si'=s+1$ and the bound~\eqref{eq-bound-filter-phi} of $\phi$ shows that
\[
\text{term~\eqref{eq-stab1}} \le C h^{2+\al} \normbig{y^0-z^0}_{s+1}.
\]
For the term~\eqref{eq-stab2} we get
\[
\text{term~\eqref{eq-stab2}} \le C h \normbig{y^0-z^0}_{s-\al} + Ch^{2+\al} \normbig{y^0-z^0}_{s+1},
\]
where we have used~\eqref{eq-bound-filter-chi} to estimate $\abs{\psi_0(\xi)}\le 1 + c \xi^{1+\al}$ for $\xi=h\om_j$, the estimate~\eqref{eq-lipschitz} from Proposition~\ref{prop-nonlinearity} with $\si=s-\al$ and $\si'=s+1$, the same estimate with $\si=\si'=s+1$ and the bound~\eqref{eq-bound-filter-phi} of $\phi$. Using in addition Lemma~\ref{lemma-reg}, we get for the term~\eqref{eq-stab3} the same estimate but with $y^1$ and $z^1$ instead of $y^0$ and $z^0$ on the right-hand side:
\[
\text{term~\eqref{eq-stab3}} \le C h \normbig{y^1-z^1}_{s-\al} + Ch^{2+\al} \normbig{y^1-z^1}_{s+1} \le 2C h \normbig{y^1-z^1}_{s+1-\al}.
\]
We then use
\[
\normbig{y^1-z^1}_{s+1-\al} \le \normbig{\cos(h\Om) \klabig{y^0-z^0} }_{s+1-\al} + h \normbig{\sinc(h\Om) \klabig{\dot{y}^0-\dot{z}^0} }_{s+1-\al} + \text{term~\eqref{eq-stab1}}
\]
and $\sinc(\xi)\le \xi^{-1}$ for $\xi>0$ to get
\[
\text{term~\eqref{eq-stab3}} \le C h \normbig{y^0-z^0}_{s+1-\al} + C h \normbig{\dot{y}^0-\dot{z}^0}_{s-\al} + C h^{3+\al}\normbig{y^0-z^0}_{s+1}. 
\]
Taking into account that $\al\le 0$, these estimates of~\eqref{eq-stab1}--\eqref{eq-stab3} prove the stability estimate of the proposition.
\end{proof}

We finally put the results of Propositions~\ref{prop-local-uncond} and~\ref{prop-stab-uncond} together to prove Theorem~\ref{thm-main} for $-1\le\al\le0$.

\begin{proof}[Proof of Theorem~\ref{thm-main} for $-1\le\al\le0$]
(a) We first consider the case $\al=0$. Let $C_1$ be the constant of Proposition~\ref{prop-local-uncond} for $\al=0$, and let $C_2$ be the constant of Proposition~\ref{prop-stab-uncond} for $\al=0$ and with $2M$ instead of $M$. We set $h_0=M/(C_1T\e^{C_2T})$.

We show, for time step-sizes $h\le h_0$, by induction on $n=0,\ldots$ that
\begin{equation}\label{eq-proof-main1-aux}
\normvbig{ \bigl(y^n,\dot{y}^n\bigr) - \bigl(y(t_n),\dot{y}(t_n)\bigr) }_{s} \le C_1 \e^{C_2 nh} nh^2
\end{equation}
as long as $t_n-t_0=nh\le T$. The case $n=0$ is clear. For $n>0$, the induction hypothesis implies for $h\le h_0$ that
\[
\normvbig{ \bigl(y^{n-1},\dot{y}^{n-1}\bigr) }_{s} \le M + C_1 \e^{C_2 T} T h \le 2M
\]
as long as $t_{n-1}-t_0=(n-1)h\le T$. This allows us to apply Propositions~\ref{prop-local-uncond} and~\ref{prop-stab-uncond} to
\begin{align*}
\normvbig{ \bigl(y^n,\dot{y}^n\bigr) - \bigl(y(t_n),\dot{y}(t_n)\bigr) }_{s} &\le
 \normvbig{ \meth\bigl(y^{n-1},\dot{y}^{n-1}\bigr) - \meth\bigl(y(t_{n-1}),\dot{y}(t_{n-1})\bigr)}_{s}\\
 &\quad + \normvbig{ \meth\bigl(y(t_{n-1}),\dot{y}(t_{n-1})\bigr) - \bigl(y(t_n),\dot{y}(t_n)\bigr) }_{s},
\end{align*}
where we denote by $\meth$ one time step with the trigonometric integrator~\eqref{eq-trigo}. Together with the induction hypothesis, this proves~\eqref{eq-proof-main1-aux} (and hence the statement of Theorem~\ref{thm-main} for $\al=0$).

(b) Now, let $-1\le\al<0$, and let $h_0$ be as above. Let further $C_1$ and $C_2$ be as above but for the new $\al$ instead of $\al=0$. We know from the above proof for the case $\al=0$ that $\normv{ (y^{n-1},\dot{y}^{n-1}) }_{s} \le 2M$ as long as $t_{n-1}-t_0\le T$. This allows us to apply Propositions~\ref{prop-local-uncond} and~\ref{prop-stab-uncond} as in part (a) of the proof to show that
\[
\normvbig{ \bigl(y^n,\dot{y}^n\bigr) - \bigl(y(t_n),\dot{y}(t_n)\bigr) }_{s-\al} \le C_1 \e^{C_2 nh} n h^{2+\al}
\]
as long as $t_n-t_0=nh\le T$. 
\end{proof}

As the above proof of Theorem~\ref{thm-main} for $\al=0$ shows, the numerical solutions stays, under the conditions of this theorem, bounded in $H^{s+1}\times H^s$,
\begin{equation}\label{eq-regularity}
\normvbig{ \bigl(y^{n},\dot{y}^{n}\bigr) }_{s} \le 2M \myfor 0\le t_n-t_0=nh \le T.
\end{equation}
This regularity of the numerical solution is essential for the proof of Theorem~\ref{thm-main} for $0<\al\le 1$ in the next subsection. Note that such an estimate cannot be obtained with the arguments of Lemma~\ref{lemma-reg} which are restricted to a bounded number of time steps.

\subsection{Proof of the higher order error bounds in lower order Sobolev spaces}

We now prove Theorem~\ref{thm-main} for $0<\al\le 1$. As in the case $-1\le\al\le 0$, we study the local error and the stability of the numerical method in Propositions~\ref{prop-local-cond} and~\ref{prop-stab-cond} below.

\begin{proposition}[Local error in $H^{s+1-\al}\times H^{s-\al}$ for $0< \al\le 1$]\label{prop-local-cond}
Let $s\ge 0$ and $0< \al\le 1$, and assume that the filter functions satisfy Assumption~\ref{assum-filter} for $\be=0$ and $\be=\al$ with constant $c$. If 
\[
\normvbig{\klabig{y(\ta),\dot{y}(\ta)}}_{s} \le M \myfor t_0\le \ta\le t_1,
\]
then
\[
\normvbig{\bigl(y(t_1),\dot{y}(t_1)\bigr) - \bigl(y^1,\dot{y}^1\bigr)}_{s-\al} \le C h^{2+\al}
\]
with a constant $C$ depending on $M$, $s$, $p$ and $c$. 
\end{proposition}
\begin{proof}
The proof is similar to the proof of Proposition~\ref{prop-local-uncond}. We denote again by $C$ a generic constant depending only on $M$, $s$, $p$ and $c$. 

(a) We use $\int_{t_0}^{t_1} (t_1-\ta) \sinc((t_1-\ta)\Om) \,\drm\ta = \frac12 h^2 \sinc^2(\frac12 h \Om)$ to
split the local error $y(t_1)-y^1$ of~\eqref{eq-local-error-y} further as follows:
\begin{subequations}
\begin{align}
y(t_1)-y^1 &= \int_{t_0}^{t_1} (t_1-\ta) \sinc\klabig{(t_1-\ta)\Om} \klaBig{ f\klabig{y(\ta)} - f\klabig{y(t_0)} } \,\drm\ta \label{eq-local-error-y-1}\\
 &\qquad + \sfrac12 h^2 \sinc^2(\sfrac12 h\Om) \klaBig{ f\klabig{y(t_0)} - f\klabig{\Phi y(t_0)} } \label{eq-local-error-y-3}\\
 &\qquad + \sfrac12 h^2 \klabig{ \sinc^2(\sfrac12 h\Om) - \Psi } f\klabig{\Phi y(t_0)}. \label{eq-local-error-y-2}
\end{align}
\end{subequations}
For the term on the right-hand side of~\eqref{eq-local-error-y-1} we get
\[
\normbig{\text{term on right-hand side of~\eqref{eq-local-error-y-1}}}_{s+1-\al} \le Ch^{2+\al},
\]
where we have used $\abs{\sinc(\xi)}\le\xi^{-1+\al}$ for $\xi>0$, the estimate~\eqref{eq-lipschitz} from Proposition~\ref{prop-nonlinearity} with $\si=s$ and $\si'=s+1$ and $y(\ta)-y(t_0)=\int_{t_0}^\ta \dot{y}(\si)\,\drm\si$. With $\abs{\sinc(\xi)}^2 \le\xi^{-1}$ for $\xi>0$, the estimate~\eqref{eq-lipschitz} from Proposition~\ref{prop-nonlinearity} with $\si=s-\al$ and $\si'=s+1$, the bound~\eqref{eq-bound-filter-phi} of $\phi$ and the bound~\eqref{eq-bound-filter-chi} of $1-\phi$, we get for the second term
\[
\normbig{\text{term~\eqref{eq-local-error-y-3}}}_{s+1-\al} \le Ch^{2+\al}.
\]
In order to estimate the last term~\eqref{eq-local-error-y-2}, we use $\abs{\sinc^2(\xi)-1} \le \xi^{\al}$, the bounds~\eqref{eq-bound-filter-phi} and~\eqref{eq-bound-filter-psi2} on $\phi$ and $1-\psi$, respectively, and the estimate~\eqref{eq-nonlinearity} from Proposition~\ref{prop-nonlinearity} with $\si'=s+1$ to get
\[
\normbig{\text{term~\eqref{eq-local-error-y-2}}}_{s+1-\al} \le C h^{2+\al}.
\]

(b) For the proof of the bound of $\dot{y}(t_1)-\dot{y}^1$ in the norm $\norm{\cdot}_{s-\al}$ we proceed similarly as in the proof of Proposition~\ref{prop-local-uncond}. We split this error again as in~\eqref{eq-local-error-doty-split}. The terms~\eqref{eq-local-error-doty-1}, \eqref{eq-local-error-doty-3} and~\eqref{eq-local-error-doty-4} are estimated in the same way as in the proof of that proposition, with the only difference that Lemma~\ref{lemma-reg} is applied with $\al=0$ instead of the $\al$ under consideration. For the quadrature error~\eqref{eq-local-error-doty-2}, we use
\[
\normbig{ \text{term~\eqref{eq-local-error-doty-2}}}_{s-\al} \le h^\al \normbig{ \text{term~\eqref{eq-local-error-doty-2}}}_{s} + h^{-1+\al} \normbig{ \text{term~\eqref{eq-local-error-doty-2}}}_{s-1}
\]
since $1\le \xi^\al + \xi^{-1+\al}$ for $\xi>0$. From the proof of Proposition~\ref{prop-local-uncond} we already know that $\norm{ \text{term~\eqref{eq-local-error-doty-2}}}_{s} \le C h^2$.
With the second-order Peano kernel $K_2(\si) = \frac12 \si (\si-1)$ of the trapezoidal rule we further get
\[
\normbig{ \text{term~\eqref{eq-local-error-doty-2}}}_{s-1} = h^3 \normbigg{\int_0^1 K_2(\si) \frac{\drm^2}{\drm t^2} f\klabig{y(t_0+\si h)} \,\drm\si}_{s-1} \le C h^3,
\]
where we have used~\eqref{eq-deriv2-nonlinearity} from Proposition~\ref{prop-nonlinearity2} in the last estimate together with the fact that $\ddot{y}=-\Om^2y+f(y)$ is bounded in the norm $\norm{\cdot}_{s-1}$. This yields 
\[
\normbig{\text{term~\eqref{eq-local-error-doty-2}}}_{s-\al} \le C h^{2+\al},
\]
and the proof of the proposition is complete.
\end{proof}

\begin{proposition}[Conditional stability in $H^{s+1-\al}\times H^{s-\al}$ for $0< \al\le 1$]\label{prop-stab-cond}
Let $s\ge 0$ and $0< \al\le 1$, and assume that the filter functions satisfy Assumption~\ref{assum-filter} for $\be=0$ with constant $c$.
We consider the trigonometric integrator~\eqref{eq-trigo} with different initial values $(y^0,\dot{y}^0)$ and $(z^0,\dot{z}^0)$. If 
\[
\normvbig{\bigl(y^0,\dot{y}^0\bigr)}_{s}\le M \myand \normvbig{\bigl(z^0,\dot{z}^0\bigr)}_{s}\le M,
\]
then 
\[
\normvbig{\bigl(y^{1},\dot{y}^1\bigr) - \bigl(z^{1},\dot{z}^1\bigr)}_{s-\al} \le  \bigl( 1 + C h \bigr) \normvbig{\bigl(y^{0},\dot{y}^0\bigr) - \bigl(z^{0},\dot{z}^0\bigr)}_{s-\al}
\]
with a constant $C$ depending on $M$, $s$, $p$ and $c$. 
\end{proposition}
\begin{proof}
As in the proof of Proposition~\ref{prop-stab-uncond}, we start from~\eqref{eq-stab}. Using~\eqref{eq-bound-filter-psi1} with $\be=0$, we estimate as in that proof
\[
\text{term~\eqref{eq-stab1}} \le C h^2 \normbig{y^0-z^0}_{s+1-\al}.
\]
Similarly, we get, using~\eqref{eq-bound-filter-chi} with $\be=0$, that
\[
\text{term~\eqref{eq-stab2}} \le C h \normbig{y^0-z^0}_{s-\al} + C h^2 \normbig{y^0-z^0}_{s+1-\al}.
\]
The same estimate holds for the term~\eqref{eq-stab3} with $y^1$ and $z^1$ on the right-hand side instead of $y^0$ and $z^0$, respectively, if we use in addition Lemma~\ref{lemma-reg} with $\al=0$. We can then argue as in the proof of Proposition~\ref{prop-stab-uncond} to replace $y^1$ and $z^1$ on the right-hand side by $y^0$ and $z^0$. This completes the proof of the stability estimate.
\end{proof}

The stability result of the previous proposition is a \emph{conditional} stability result, since it requires regularity in a higher Sobolev space than the one in which stability is shown. In the following proof of Theorem~\ref{thm-main} for $0<\al\le1$, we can afford this higher regularity of the numerical solution, since our analysis of the previous subsection implies this regularity, see in particular~\eqref{eq-regularity}. Nevertheless, we mention that there are some special cases in which the above conditional stability result can be turned into an unconditional stability result, for example for $s\ge 1$ (or even $s>\frac12$) by virtue of~\eqref{eq-lipschitz} from Proposition~\ref{prop-nonlinearity}, or for $p=2$ by virtue of part (ii) of Proposition~\ref{prop-algebra} and a slightly stronger assumption on $\psi$.

\begin{proof}[Proof of Theorem~\ref{thm-main} for $0<\al\le1$]
The proof is the same as the one for $-1\le\al<0$ in the previous subsection. Of central importance is the fact that we know from the analysis there that the numerical solution is bounded in $H^{s+1}\times H^s$, see~\eqref{eq-regularity}. Together with the boundedness~\eqref{eq-finiteenergy} of the exact solution in $H^{s+1}\times H^s$, this ensures that the regularity assumptions for the stability estimate of Proposition~\ref{prop-stab-cond} are fulfilled.
\end{proof}

\subsection{On the use of a filter inside the nonlinearity}\label{subsec-useoffilter}

After having completed the proof of Theorem~\ref{thm-main} in the previous subsection, we comment in this subsection on the filter $\Phi$ and give an outline of a slightly different proof of Theorem~\ref{thm-main}. 

We consider the trigonometric integrator~\eqref{eq-trigo}, which uses a filter $\Phi$ inside the nonlinearity $f$, applied to~\eqref{eq-nlw-semi}. This method can be written as
\begin{equation}\label{eq-trigo-mod}
\begin{pmatrix} z^{n+1}\\ \dot{z}^{n+1}\end{pmatrix} = R(h) \begin{pmatrix} z^n\\ \dot{z}^n\end{pmatrix} + \begin{pmatrix} \sfrac12 h^2 \Psi \widetilde{f}(z^n)\\ \sfrac12 h \Psi_0 \widetilde{f}(z^n) + \sfrac12 h \Psi_1 \widetilde{f}(z^{n+1})\end{pmatrix}
\end{equation}
with 
\begin{equation}\label{eq-useoffilter-aux}
\klabig{ z^n,\dot{z}^n } = \klabig{ y^n,\dot{y}^n }
\end{equation}
and the modified nonlinearity
\[
\widetilde{f}(z) = f \klabig{ \Phi z }.
\]
This is a trigonometric integrator, with filters $\Psi$, $\Psi_0$ and $\Psi_1$ but no filter inside the nonlinearity,
applied to the system
\[
\ddot{z}(t) = -\Om^2 z(t) + \widetilde{f} \klabig{ z(t) }, \qquad z(t_0) = y(t_0), \quad \dot{z}(t_0) = \dot{y}(t_0).
\]

On the other hand, we have, under the assumptions~\eqref{eq-bound-filter-phi} and~\eqref{eq-bound-filter-chi} on $\phi$ with $\be=0$ and $\be =\al$ and under the assumption~\eqref{eq-finiteenergy} on $(y(t),\dot{y}(t))$, that
\[
\normvbig{ \klabig{ y(t)-z(t), \dot{y}(t)-\dot{z}(t) } }_{s-\al} \le C h^{1+\al} \myfor 0\le t-t_0\le T
\]
for $-1\le \al\le 1$. 
Instead of giving the full details here, we only mention that this estimate can be shown with the arguments used in the proofs of the stability estimates of Propositions~\ref{prop-stab-uncond} and~\ref{prop-stab-cond} and with the Gronwall inequality applied to the variation-of-constants formula~\eqref{eq-voc} for $(y-z,\dot{y}-\dot{z})$ together with a bootstrap argument; again, one has to consider first the case $\al=0$ and then the case of a general $\al$. 
From~\eqref{eq-useoffilter-aux}, we then infer
\[
\absBig{ \normvbig{ \klabig{ y(t_n)-y^n, \dot{y}(t_n)-\dot{y}^n } }_{s-\al} - \normvbig{ \klabig{ z(t_n)-z^n, \dot{z}(t_n)-\dot{z}^n } }_{s-\al} } \le C h^{1+\al}.
\]

Hence, the trigonometric integrator~\eqref{eq-trigo} with filters $\Psi$, $\Psi_0$, $\Psi_1$ and $\Phi$ is in $H^{s+1-\al}\times H^{s-\al}$ of order $1+\al$ if and only if the same holds for the trigonometric integrator~\eqref{eq-trigo-mod} with filters $\Psi$, $\Psi_0$, $\Psi_1$ and $\Id$. 
This shows that it would be sufficient to consider the case $\Phi=\Id$ in the proof of Theorem~\ref{thm-main}. It also shows that the filter $\Phi$ is not important for the sake of proving such error bounds. This latter conclusion does not hold for Gautschi-type methods, for which numerical experiments suggest that a suitably chosen filter $\Phi$ is necessary to have optimal temporal error bounds. This latter conclusion neither holds for the equations considered in~\cite{GarciaArchilla1999,Grimm2006,Hairer2006}.

\section{Extensions}\label{sec-ext}

Revisiting the proof of Theorem~\ref{thm-main} as given in the previous section shows that only the following properties of the diagonal matrix $\Om=\diag(\om_j)_{j\in\disc}$ and the nonlinearity $f$ in~\eqref{eq-nlw-semi} are needed.
\begin{itemize}
\item The frequencies $\om_j$ behave like $\abs{j}$: there exist positive constants $c_1$ and $c_2$ such that
\begin{equation}\label{eq-om-cond}
c_1 \abs{j} \le \om_j \le c_2 (1+\abs{j}), \qquad j\in\disc.
\end{equation}
\item The nonlinearity has the properties of Propositions~\ref{prop-nonlinearity} and~\ref{prop-nonlinearity2}.
\end{itemize}
The norm $(\sum_{j\in\disc} \max(\om_j,\om_{\min})^{2 s} \abs{y_j}^{2})^{1/2}$, where $\om_{\min}$ denotes the minimal nonzero frequency, is then equivalent to the norm $\norm{\cdot}_s$, and the proof of Theorem~\ref{thm-main} transfers with this norm to such situations. 
The statement of Theorem~\ref{thm-main} thus holds (with constants depending in addition on $c_1$ and $c_2$ from~\eqref{eq-om-cond}) for trigonometric integrators applied to general equations of the form~\eqref{eq-nlw-semi} that satisfy these two conditions. We illustrate this on some examples.

\subsection{Error bounds for more general nonlinearities}\label{subsec-generalnonlinear}

Let $g\colon\C\rightarrow\C$ be an analytic function with $g(0)=0$ and $g'(0)\le0$, given by
\[
g(x) = \sum_{m=1}^\infty a_m x^m.
\]
We consider the nonlinear wave equation
\begin{equation}\label{eq-nlw-analytic}
u_{tt} - u_{xx} = g(u), \qquad u=u(x,t)
\end{equation}
with this nonlinearity. This includes the pure power nonlinear wave equation~\eqref{eq-nlw} that we have considered so far ($g(x)=x^p$), but also the nonlinear Klein--Gordon equation
\[
u_{tt} - u_{xx} + \rho u = u^p, \qquad \rho>0,
\]
where $g(x) = -\rho x + x^p$, and the Sine--Gordon equation
\[
u_{tt} - u_{xx} = -\sin(u),
\]
where $g(x)=-\sin(x)$.

The discretization in space of this equation by spectral collocation can be done in the same way as in Subsection~\ref{subsec-collocation}. This leads to an equation of the form~\eqref{eq-nlw-semi} with the frequencies
\[
\om_j = \sqrt{j^2 - g'(0)}
\]
and the nonlinearity
\[
f(y) = \sum_{m=2}^\infty a_m \bigl(\underbrace{y\conv \dots\conv y}_{m\text{ times}}\bigr).
\]
The new frequencies $\om_j$ satisfy~\eqref{eq-om-cond} with $c_1=1$ and $c_2=1-g'(0)$. The analyticity of $g$ then allows us to extend Propositions~\ref{prop-nonlinearity} and~\ref{prop-nonlinearity2} from pure power nonlinearities of the form $y\conv \dotsm \conv y$ to the above nonlinearity $f$. 

Hence, the error bounds of Theorem~\ref{thm-main} extend to trigonometric integrators applied to the spectral semi-discretization in space of the more general nonlinear wave equation~\eqref{eq-nlw-analytic} instead of~\eqref{eq-nlw}. Similarly, one can consider nonlinear wave equations of the form $u_{tt} - u_{xx} = g(\abs{u}^2) u$ with complex valued solutions.

\subsection{Error bounds for the spatial semi-discretization by finite differences}\label{subsec-findiff}

For the spatial discretization by finite differences (instead of spectral collocation), one replaces the derivative $u_{xx}(x,t)$ in the nonlinear wave equation~\eqref{eq-nlw} by the difference
\[
\frac{u(x+\deltax,t) - 2u(x,t) + u(x-\deltax,t)}{(\deltax)^2} \with \deltax=\frac{\pi}{K}.
\]
Then one inserts the points $x_k=\pi k/K$ in the equation. 

As in the case of the spectral collocation method of Subsection~\ref{subsec-collocation}, we define the vector $y=(y_j)_{j\in\disc}$ by $u(x_k,t)=\sum_{j\in\disc} y_j(t)\e^{\iu jx_k}$, $k\in\disc$. This then leads again to a system of the form~\eqref{eq-nlw-semi} with exactly the same nonlinearity as in Subsection~\ref{subsec-collocation}. The only difference compared to~\eqref{eq-nlw-semi} is that the frequencies $\om_j$ now read
\[
\om_j = \frac{2}{\deltax} \, \absBig{\sin\Bigl(\frac{j\deltax}{2}\Bigr)}.
\]
These frequencies satisfy~\eqref{eq-om-cond} with $c_1=2/\pi$ and $c_2=1$. 

Theorem~\ref{thm-main} thus also holds if the spatial semi-discretization by finite differences instead of spectral collocation is considered. It is interesting to observe that the finite difference semi-discretization in space requires higher regularity assumptions on the exact solution for convergence than the semi-discretization in time by trigonometric integrators.

\subsection{Error bounds for the St\"ormer--Verlet/leapfrog discretization in time}\label{subsec-sv}

The popular St\"ormer--Verlet/leapfrog discretization in time of the spatially discrete wave equation~\eqref{eq-nlw-semi} reads 
\begin{equation}\label{eq-sv}
y^{n+1} - 2y^n + y^{n-1} = -h^2\Om^2 y^n + h^2 f(y^n)
\end{equation}
with starting approximation $y^1 = y^0 + h \dot{y}^0 - \sfrac12 h^2\Om^2 y^0 + \sfrac12 h^2 f(y^0)$ and velocity approximation $2h \dot{y}^n = y^{n+1} - y^{n-1}$, see, for instance, \cite[Section~XIII.8]{Hairer2006}. 

Under the CFL-type step-size restriction $h\om_j< 2$ for all $j\in\disc$, i.e., $hK< 2$,
this method can be interpreted as a trigonometric integrator for an equation with modified frequencies, see again~\cite[Section~XIII.8]{Hairer2006}. Indeed, under this step-size restriction, one can introduce modified frequencies $0\le \widetilde{\om}_j< h^{-1}\pi$ by
\[
\widetilde{\Om} = \diag(\widetilde{\om}_j)_{j\in\disc} \with \cos\klabig{h\widetilde{\om}_j} = 1-\sfrac12 h^2\om_j^2
\]
and modified velocities
\[
\dot{\widetilde{y}} = \sinc\klabig{h\widetilde{\Om}}^{-1} \dot{y}.
\]
The St\"ormer--Verlet/leapfrog discretization~\eqref{eq-sv} then takes the form
\begin{equation}\label{eq-sv-trigo}
\begin{pmatrix} y^{n+1}\\ \dot{\widetilde{y}}^{n+1}\end{pmatrix} = \widetilde{R}(h) \begin{pmatrix} y^n\\ \dot{\widetilde{y}}^n\end{pmatrix} + \begin{pmatrix} \sfrac12 h^2 \Psi f(\Phi y^n)\\ \sfrac12 h \Psi_0 f(\Phi y^n) + \sfrac12 h \Psi_1 f(\Phi y^{n+1})\end{pmatrix},
\end{equation}
where $\widetilde{R}$ is the resolvent $R$ of~\eqref{eq-R} but with the modified frequencies $\widetilde{\Om}$ instead of $\Om$, and where
\begin{equation}\label{eq-sv-filter}
\Phi = \Psi = \Id, \qquad \Psi_0 = \cos\klabig{h\widetilde{\Om}} \sinc\klabig{h\widetilde{\Om}}^{-1}, \qquad \Psi_1 = \sinc\klabig{h\widetilde{\Om}}^{-1}.
\end{equation}
In this sense, the St\"ormer--Verlet/leapfrog discretization~\eqref{eq-sv} can be considered as a trigonometric integrator applied to the system
\begin{equation}\label{eq-svmod}
\ddot{z}(t) = -\widetilde{\Om}^2 z(t) + f \klabig{ z(t) }, \qquad z(t_0) = y(t_0), \quad \dot{z}(t_0) = \dot{\widetilde{y}}^0 = \sinc(h\widetilde{\Om}^{-1}) \dot{y}(t_0).
\end{equation}
This leads to the following convergence result.

\begin{theorem}\label{thm-sv}
Let $s\ge 0$ and $-1\le\al\le \min(1,\sfrac23 s+\sfrac13)$, and assume that the exact solution $(y(t),\dot{y}(t))$ of the spatial semi-discretization~\eqref{eq-nlw-semi} of the nonlinear wave equation~\eqref{eq-nlw} as well as the exact solution $(z(t),\dot{z}(t))$ of the equation~\eqref{eq-svmod} with modified frequencies and modified initial values both satisfy the finite energy assumption~\eqref{eq-finiteenergy} of Theorem~\ref{thm-main}. 

Then, there exists $h_0>0$ such that for all time step-sizes $h\le h_0$ that fulfill the step-size restriction
\begin{equation}\label{eq-cfl}
h K \le c_0 < 2,
\end{equation}
the following error bound holds for the numerical solution $(y^n,\dot{y}^n)$ computed with the St\"ormer--Verlet/leapfrog method~\eqref{eq-sv}:
\[
\normbig{y(t_n) - y^n}_{s+1-3(1+\al)/2} + \normbig{\dot{y}(t_n) - \dot{y}^n}_{s-3(1+\al)/2} \le C h^{1+\al} \quad\text{for}\quad 0\le t_n-t_0 \le T. 
\]
The constants $C$ and $h_0$ depend only on $M$ and $s$ from~\eqref{eq-finiteenergy}, the power $p$ of the nonlinearity in~\eqref{eq-nlw}, the final time $T$ and the constant $c_0$ from~\eqref{eq-cfl}.
\end{theorem}
\begin{proof}
We decompose the errors as
\begin{align*}
y(t_n) - y^n &= \klabig{y(t_n)-z(t_n)} + \klabig{z(t_n)-y^n},\\
\dot{y}(t_n) - \dot{y}^n &= \klabig{\dot{y}(t_n)-\dot{z}(t_n)} + \klabig{\dot{z}(t_n)-\dot{\widetilde{y}}^n} + \klabig{\dot{\widetilde{y}}^n-\dot{y}^n}
\end{align*}
and estimate the terms separately. By $C$, we denote a generic constant depending only $M$, $s$, $p$, $T$ and $c_0$.

(a) Error of the trigonometric integrator for the modified equation.
By Taylor expansion, we have
\begin{equation}\label{eq-om-omtilde}
h^2 \absbig{\om_j^2-\widetilde{\om}_j^2} \le \sfrac1{12} h^4 \widetilde{\om}_j^4 \myfor j\in\disc.
\end{equation}
Since the modified frequencies satisfy $h\widetilde{\om}_j\le \pi$ for all $j\in\disc$, this implies
\begin{equation}\label{eq-boundomtilde}
c_1 \om_j \le \widetilde{\om}_j \le c_2 \om_j \myfor j\in\disc
\end{equation}
with $c_1 = 1/(1+\pi^2/12)^{1/2}$ and $c_2 = 1/(1-\pi^2/12)^{1/2}$.
This shows that the frequencies of the system~\eqref{eq-svmod} for $(z,\dot{z})$ satisfy~\eqref{eq-om-cond}. Moreover, the step-size restriction~\eqref{eq-cfl} ensures that $h\widetilde{\om}_j$ is bounded away from $\pi$, and hence Assumption~\ref{assum-filter} on the filter functions holds for the filters~\eqref{eq-sv-filter} for all $-1\le\be\le 1$ with a constant $c$ depending only on $c_0$. We may thus apply Theorem~\ref{thm-main} to the trigonometric integrator~\eqref{eq-sv-trigo} applied to~\eqref{eq-svmod}. This shows that
\begin{equation}\label{eq-sv-errortrigo}
\normvbig{\klabig{z(t_n)-y^n,\dot{z}(t_n)-\dot{\widetilde{y}}^n}}_{s-\al} \le C h^{1+\al} \myfor 0\le t_n-t_0\le T,
\end{equation}
where we use the norm $\normv{\cdot}_{\si}$ of Subsection~\ref{subsec-proof1}.

(b) Error from modifying the velocities.
From the error bound~\eqref{eq-sv-errortrigo} we get $\norm{\dot{\widetilde{y}}^n}_{s}\le C$, and from~\eqref{eq-boundomtilde} we get $\abs{1-\sinc(h\widetilde{\om}_j)} \le C h^{1+\al} \om_j^{1+\al}$. This shows that
\begin{equation}\label{eq-sv-errorvelocity}
\normbig{ \dot{\widetilde{y}}^n - \dot{y}^n }_{s-1-\al} \le C h^{1+\al}.
\end{equation}

(c) Error from modifying the frequencies and initial values.
The solution $(z,\dot{z})$ of~\eqref{eq-svmod} can be expressed by the same variation-of-constants formula~\eqref{eq-voc} as the solution $(y,\dot{y})$ of~\eqref{eq-nlw-semi}, but with $\widetilde{R}$ instead of $R$ (and $z$ instead of $y$, of course). Subtracting these formulas gives
\begin{subequations}
\begin{align}
\begin{pmatrix} y(t)-z(t)\\ \dot{y}(t)-\dot{z}(t) \end{pmatrix} &= 
R(t-t_0) \begin{pmatrix} y(t_0)-z(t_0)\\ \dot{y}(t_0)-\dot{z}(t_0) \end{pmatrix}\label{eq-proof-sv-aux1}\\
&\qquad + \klabig{ R(t-t_0) - \widetilde{R}(t-t_0) } \begin{pmatrix} z(t_0)\\ \dot{z}(t_0) \end{pmatrix}\label{eq-proof-sv-aux2}\\
&\qquad + \int_{t_0}^t R(t-\ta) \begin{pmatrix} 0\\ f(y(\ta)) - f(z(\ta)) \end{pmatrix} \,\drm\ta\label{eq-proof-sv-aux3}\\
&\qquad + \int_{t_0}^t \klabig{ R(t-\ta) - \widetilde{R}(t-\ta) } \begin{pmatrix} 0\\ f(z(\ta)) \end{pmatrix} \,\drm\ta.\label{eq-proof-sv-aux4}
\end{align}
\end{subequations}
We estimate the terms on the right-hand side separately. 
Form the fact that $R$ almost preserves the norm $\normv{\cdot}_{\si}$ (see~\eqref{eq-normR}) and from~\eqref{eq-sv-errorvelocity}, we get
\[
\normvbig{\text{term on right-hand side of~\eqref{eq-proof-sv-aux1}}}_{s-1-\al} \le C h^{1+\al}.
\]
Similarly, we get
\[
\normvbig{\text{term~\eqref{eq-proof-sv-aux3}}}_{s+1-3(1+\al)/2} \le C \int_{t_0}^t \norm{y(\ta)-z(\ta)}_{s+1-3/2(1+\al)} \,\drm\ta,
\]
where we have used in addition~\eqref{eq-lipschitz} from Proposition~\ref{prop-nonlinearity} with $\si=s+1-3(1+\al)/2$ and $\si'=s+1$ (note that $\si\ge -1$ since we assume that $\al\le 2s/3+1/3$). 
In order to estimate the terms~\eqref{eq-proof-sv-aux2} and~\eqref{eq-proof-sv-aux4}, we study $R(t) - \widetilde{R}(t)$ for $0\le t\le T$. Using the trigonometric identity $\cos(a) - \cos(b) = 2 \sin((a+b)/2) \sin((b-a)/2)$ and $\abs{\sin((a+b)/2)}\le 1$, we get
\[
\absbig{\cos(t\om_j) - \cos(t\widetilde{\om}_j)}  \le C h^{1+\al} \, \om_j^{3(1+\al)/2} \myfor j\in\disc,
\]
where we have distinguished between $1\le h^2\om_j^3$ and $h^2\om_j^3\le1$; in the first case, we use $\abs{\sin((b-a)/2)}\le 1$ and $ 1 \le h^{1+\al} \, \om_j^{3(1+\al)/2}$, whereas we use $\abs{\sin((b-a)/2)}\le \abs{b-a}$, $\abs{\om_j-\widetilde{\om}_j}\le C h^2\om_j^3$ by~\eqref{eq-om-omtilde} and~\eqref{eq-boundomtilde} and $h^{1-\al} \, \om_j^{3(1-\al)/2}\le 1$ in the second case. Similarly, we get with $\sin(a) - \sin(b) = 2 \cos((a+b)/2) \sin((a-b)/2)$ that 
\[
\absbig{\sin(t\om_j) - \sin(t\widetilde{\om}_j)}  \le C h^{1+\al} \, \om_j^{3(1+\al)/2} \myfor j\in\disc.
\]
Using $h\om_j\le 2$, we also obtain from~\eqref{eq-om-omtilde} and~\eqref{eq-boundomtilde} that
\[
\absbigg{\frac{\om_j-\widetilde{\om}_j}{\om_j}} \le C h^{1+\al} \, \om_j^{1+\al} \myfor j\in\disc.
\]
These estimates show that
\[
\normvbig{\text{term~\eqref{eq-proof-sv-aux2}}}_{s-3(1+\al)/2} \le C h^{1+\al}
\]
since $\normv{(z(t_0),\dot{z}(t_0))}_s\le M$,
and similarly that
\[
\normvbig{\text{term~\eqref{eq-proof-sv-aux4}}}_{s+1-3(1+\al)/2} \le C h^{1+\al},
\]
since $\normv{(0,f(z(\ta)))}_{s+1} \le M$ by~\eqref{eq-nonlinearity} from Proposition~\ref{prop-nonlinearity} with $\si=\si'=s+1$. Taking the estimates of the different terms~\eqref{eq-proof-sv-aux1}--\eqref{eq-proof-sv-aux4} together shows that, for $0\le t-t_0\le T$,
\[
\normvbig{\klabig{y(t)-z(t),\dot{y}(t)-\dot{z}(t)}}_{s-3(1+\al)/2} \le C h^{1+\al} + C \int_{t_0}^t \norm{y(\ta)-z(\ta)}_{s+1-3/2(1+\al)} \,\drm\ta.
\]
The Gronwall inequality then implies a bound by $Ch^{1+\al}$ of the difference $(y(t)-z(t),\dot{y}(t)-\dot{z}(t))$ in $H^{s+1-3(1+\al)/2}\times H^{s-3(1+\al)/2}$. Together with the estimates~\eqref{eq-sv-errortrigo} and~\eqref{eq-sv-errorvelocity} of parts (a) and (b) of the proof, respectively, this completes the proof of the theorem.
\end{proof}

For $s=0$, for example, the above theorem gives for the St\"ormer--Verlet/leapfrog discretization uniform convergence of order $2/3$ in $H^0\times H^{-1}$ (with $\al = -1/3$). This order of convergence has also been observed in the numerical experiment of Subsection~\ref{subsec-numexp}, see Figure~\ref{fig-stoermerverlet}. This is in striking contrast to trigonometric integrators that are in this situation second-order convergent, see Theorem~\ref{thm-main}.
In comparison with trigonometric integrators, the St\"ormer--Verlet/leapfrog discretization in time thus not only requires the CFL-type step-size restriction~\eqref{eq-cfl}, but it also converges only in Sobolev spaces of comparatively low order.

\section{Conclusion}\label{sec-conclusion}

An error analysis of trigonometric integrators applied to spatial semi-discretizations of some semilinear wave equations has been given. The analysis is uniform in the spatial discretization parameter, and it extends in a straightforward way to the spatially continuous semi-discretization in time by trigonometric integrators. In contrast to previous works on error bounds for these integrators, the presented analysis takes care and makes use of the structure of nonlinearity in the scale of Sobolev spaces. 

The flexibility of the presented error analysis has been illustrated by its extension to more general nonlinearities, to spatial semi-discretizations by finite differences and to the St\"ormer--Verlet/leapfrog discretization in time.
Likewise, we expect that an extension to multiple space dimensions is possible. Challenging problems for future work are the study of related questions in the case of quasilinear wave equations and the explanation of the remarkably good behaviour of Deuflhard's method that we have observed in numerical experiments.

\subsection*{Acknowledgement}

I thank Christian Lubich (Universit\"at T\"ubingen) for pointing out that the St\"ormer--Verlet/leapfrog discretization is covered by the presented error analysis, which led to Subsection~\ref{subsec-sv}. 
This work was partially supported by DFG project GA 2073/2-1.

\end{document}